\documentclass[10pt, a4paper, reqno]{amsart}
\usepackage[utf8]{inputenc}
\usepackage{mathrsfs}
\usepackage{bm}
\usepackage{amstext}
\usepackage{amsthm}
\usepackage{amssymb}
\usepackage{mathtools}
\usepackage[all]{xy}
\usepackage{cleveref}
\usepackage{graphics}
\usepackage{mathdots}
\makeatletter
\numberwithin{equation}{section}
\numberwithin{figure}{section}
\theoremstyle{plain}
\newtheorem{thm}{Theorem}[section]
  \crefname{thm}{Theorem}{Theorems}
  \newtheorem{lem}[thm]{Lemma}
  \crefname{lem}{Lemma}{Lemmas}
  \newtheorem{prop}[thm]{Proposition}
  \crefname{prop}{Proposition}{Propositions}
  \newtheorem{cor}[thm]{Corollary}
	\crefname{cor}{Corollary}{Corollaries}
  \newtheorem*{theor1}{Theorem 1}
  \newtheorem*{theor2}{Theorem 2}
  \newtheorem*{ack}{Acknowledgments}
\theoremstyle{definition}
  \newtheorem{defi}[thm]{Definition}
  \crefname{defi}{Definition}{Definitions}
  \theoremstyle{remark}
  
  \crefname{ntn}{Notation}{Notations}
	 \theoremstyle{remark}
  \newtheorem{rem}[thm]{Remark}
  \crefname{rem}{Remark}{Remarks}
  \newtheorem{ex}[thm]{Example}
  \crefname{ex}{Example}{Examples}
\makeatother
\usepackage{a4wide}
\def\r{\mathbb{R}}
\def\c{\mathbb{C}}

\def\z{\mathbb{Z}}
\makeatletter
    
    \@addtoreset{equation}{section}
 \makeatother

\title[Newton--Okounkov polytopes and marked chain-order polytopes]{\fontsize{11pt}{11pt}\selectfont Newton--Okounkov polytopes of flag varieties and marked chain-order polytopes}

\date{}

\author{Naoki FUJITA}

\address[Naoki FUJITA]{Graduate School of Mathematical Sciences, The University of Tokyo, 3-8-1 Komaba, Meguro-ku, Tokyo 153-8914, Japan.}

\email{nfujita@ms.u-tokyo.ac.jp}



\thanks{This work was supported by Grant-in-Aid for JSPS Fellows (No.\ 19J00123) and by JSPS Grant-in-Aid for Early-Career Scientists (No.\ 20K14281).} 

\begin{document}
\begin{abstract}
Marked chain-order polytopes are convex polytopes constructed from a marked poset, which give a discrete family relating a marked order polytope with a marked chain polytope. 
In this paper, we consider the Gelfand--Tsetlin poset of type $A$, and realize the associated marked chain-order polytopes as Newton--Okounkov bodies of the flag variety. 
Our realization connects previous realizations of Gelfand--Tsetlin polytopes and Feigin--Fourier--Littelmann--Vinberg polytopes as Newton--Okounkov bodies in a uniform way. 
As an application, we prove that the flag variety degenerates into the irreducible normal projective toric variety corresponding to a marked chain-order polytope. 
We also construct a specific basis of an irreducible highest weight representation which is naturally parametrized by the set of lattice points in a marked chain-order polytope. 
\end{abstract}

\maketitle
\tableofcontents 

\section{Introduction}\label{s:intro}

A Newton--Okounkov body $\Delta(X, \mathcal{L}, \nu)$ is a convex body defined from a projective variety $X$ with a globally generated line bundle $\mathcal{L}$ on $X$ and with a higher rank valuation $\nu$ on the function field $\c(X)$. 
It was originally introduced by Okounkov \cite{Oko1, Oko2, Oko3} to study multiplicity functions for representations of a reductive group, and afterward developed independently by Lazarsfeld--Mustata \cite{LM} and by Kaveh--Khovanskii \cite{KK1, KK2} as a generalization of the notion of Newton polytopes for toric varieties to arbitrary projective varieties. 
A remarkable application is that the theory of Newton--Okounkov bodies gives a systematic construction of toric degenerations \cite{And} and completely integrable systems \cite{HK}. 
In the present paper, we restrict ourselves to the case of flag varieties.
In this case, the following specific polytopes can be realized as Newton--Okounkov bodies:
\begin{enumerate}
\item[(i)] Gelfand--Tsetlin polytopes \cite{Oko1, Oko2},
\item[(ii)] Berenstein--Littelmann--Zelevinsky's string polytopes \cite{Kav},
\item[(iii)] Nakashima--Zelevinsky polytopes \cite{FN},
\item[(iv)] Lusztig polytopes \cite{FaFL},
\item[(v)] Feigin--Fourier--Littelmann--Vinberg (FFLV) polytopes \cite{FeFL3, Kir1, Kir2},
\item[(vi)] polytopes constructed from extended $g$-vectors in cluster theory \cite{FO2},
\end{enumerate}
where the attached references are the ones giving such realizations.
Our aim of the present paper is to add a new family of convex polytopes to this list, that is, we realize the marked chain-order polytopes associated with the Gelfand--Tsetlin poset of type $A$ as Newton--Okounkov bodies of the flag variety. 

A marked chain-order polytope $\Delta_{\mathcal{C}, \mathcal{O}}(\Pi, \Pi^\ast, \lambda)$ is a convex polytope constructed from a marked poset $(\Pi, \Pi^\ast, \lambda)$ with a partition $\Pi \setminus \Pi^\ast = \mathcal{C} \sqcup \mathcal{O}$ under the assumption that $\Pi$ is finite and that $\Pi^\ast$ contains all maximal and minimal elements in $\Pi$.
It was introduced by Fang--Fourier \cite{FF} in some specific setting and by Fang--Fourier--Litza--Pegel \cite{FFLP} in general setting as a family relating the marked order polytope $\mathcal{O}(\Pi, \Pi^\ast, \lambda)$ with the marked chain polytope ${\mathcal C}(\Pi, \Pi^\ast, \lambda)$.
We consider this framework when $(\Pi, \Pi^\ast, \lambda)$ is the Gelfand--Tsetlin poset $(\Pi_A, \Pi_A^\ast, \lambda)$ of type $A$. 

To state our main result explicitly, let $G = SL_{n+1}(\c)$, $B \subseteq G$ the subgroup of upper triangular matrices, $W$ the Weyl group, and $P_+$ the set of dominant integral weights.
We denote by $X(w) \subseteq G/B$ the Schubert variety corresponding to $w \in W$, by $e \in W$ the identity element, and by $w_0 \in W$ the longest element. 
Note that the full flag variety $G/B$ coincides with the Schubert variety $X(w_0)$ corresponding to $w_0$. 
Let ${\bm i}_A = (i_1, i_2, \ldots, i_N)$ be a reduced word for $w_0$ given in \eqref{eq:reduced_word}, and set $w_{\geq k} \coloneqq s_{i_k} s_{i_{k+1}} \cdots s_{i_N}$ for $1 \leq k \leq N$. 
For $\lambda \in P_+$, denote by $\mathcal{L}_\lambda$ the corresponding globally generated line bundle on $X(w)$, by $V(\lambda)$ the irreducible highest weight $G$-module with highest weight $\lambda$, and by $(\Pi_A, \Pi_A^\ast, \lambda)$ the associated Gelfand--Tsetlin poset (see Figure \ref{type_A_marked_Hasse}). 
Then the marked order polytope $\mathcal{O}(\Pi_A, \Pi_A^\ast, \lambda)$ coincides with the Gelfand--Tsetlin polytope $GT(\lambda)$, which was introduced by Gelfand--Tsetlin \cite{GT} to parametrize a specific basis of $V(\lambda)$, called the Gelfand--Tsetlin basis. 
The Gelfand--Tsetlin polytope $GT(\lambda)$ can be realized as a Newton--Okounkov body $\Delta(G/B, \mathcal{L}_\lambda, \nu_{GT})$ up to translations (see \cite{Oko1, FO}), where $\nu_{GT}$ is a valuation given by counting the orders of zeros/poles along the following sequence of Schubert varieties:  
\[(G/B =)\ X(w_{\geq 1}) \supsetneq X(w_{\geq 2}) \supsetneq \cdots \supsetneq X(w_{\geq N}) \supsetneq X(e)\ (= B/B).\]
In addition, the marked chain polytope ${\mathcal C}(\Pi_A, \Pi_A^\ast, \lambda)$ coincides with the FFLV polytope $FFLV(\lambda)$, which was introduced by Feigin--Fourier--Littelmann \cite{FeFL1} and Vinberg \cite{Vin} to study the PBW (Poincar\'{e}--Birkhoff--Witt) filtration on $V(\lambda)$. 
Kiritchenko \cite{Kir1} proved that the FFLV polytope $FFLV(\lambda)$ coincides with the Newton--Okounkov body $\Delta(G/B, \mathcal{L}_\lambda, \nu_{FFLV})$ associated with a valuation $\nu_{FFLV}$ given by counting the orders of zeros/poles along a sequence of translated Schubert varieties (see \eqref{eq:definition_of_full_translation} for the precise definition). 
Feigin--Fourier--Littelmann \cite{FeFL3} also realized $FFLV(\lambda)$ as a Newton--Okounkov body $\Delta(G/B, \mathcal{L}_\lambda, \nu)$ using a different kind of valuation. 
In the present paper, we consider a sequence $X_\bullet^{\mathcal{C}, \mathcal{O}}$ of partially translated Schubert varieties defined from a partition $\Pi_A \setminus \Pi^\ast_A = \mathcal{C} \sqcup \mathcal{O}$; see \eqref{eq:definition_of_partial_translation} for the precise definition. 
The following is the main result of this paper. 

\begin{theor1}[{\cref{t:main_thm}}]
For each partition $\Pi_A \setminus \Pi^\ast_A = \mathcal{C} \sqcup \mathcal{O}$ and $\lambda \in P_+$, the marked chain-order polytope $\Delta_{\mathcal{C}, \mathcal{O}}(\Pi_A, \Pi^\ast_A, \lambda)$ coincides with a Newton--Okounkov body $\Delta(G/B, \mathcal{L}_\lambda, \nu_{\mathcal{C}, \mathcal{O}}^{\rm low})$ up to translations by integer vectors, where $\nu_{\mathcal{C}, \mathcal{O}}^{\rm low}$ is a valuation given by counting the orders of zeros/poles along the sequence $X_\bullet^{\mathcal{C}, \mathcal{O}}$.
\end{theor1}

Theorem 1 connects the previous realizations of Gelfand--Tsetlin polytopes \cite{Oko1, FO} and FFLV polytopes \cite{Kir1} as Newton--Okounkov bodies in a uniform way. 
As an application of Theorem 1, we construct a specific $\c$-basis of $V(\lambda)$ which is naturally parametrized by the set of lattice points in $\Delta_{\mathcal{C}, \mathcal{O}}(\Pi_A, \Pi^\ast_A, \lambda)$ (see \cref{t:explicit_basis}).
This basis is an analog of an \emph{essential basis} introduced by Feigin--Fourier--Littelmann \cite{FeFL1, FeFL3} and by Fang--Fourier--Littelmann \cite{FaFL}.
Let $P_{++}$ denote the set of regular dominant integral weights.
We also use Theorem 1 to prove that Anderson's construction \cite{And} of toric degenerations can be applied to $\Delta_{\mathcal{C}, \mathcal{O}}(\Pi_A, \Pi^\ast_A, \lambda)$ as follows. 

\begin{theor2}[{\cref{t:toric_deg}}]
For each partition $\Pi_A \setminus \Pi^\ast_A = \mathcal{C} \sqcup \mathcal{O}$ and $\lambda \in P_{++}$, there exists a flat degeneration of $G/B$ to the irreducible normal projective toric variety corresponding to the marked chain-order polytope $\Delta_{\mathcal{C}, \mathcal{O}}(\Pi_A, \Pi^\ast_A, \lambda)$.
\end{theor2}

Our proof of Theorem 2 shows that we can apply Harada--Kaveh's theory \cite{HK} to the marked chain-order polytope $\Delta_{\mathcal{C}, \mathcal{O}}(\Pi_A, \Pi^\ast_A, \lambda)$ for each partition $\Pi_A \setminus \Pi^\ast_A = \mathcal{C} \sqcup \mathcal{O}$ and $\lambda \in P_{++}$, which implies that there exists a completely integrable system on $G/B$ whose image coincides with $\Delta_{\mathcal{C}, \mathcal{O}}(\Pi_A, \Pi^\ast_A, \lambda)$.
We also study the highest term valuation $\nu_{\mathcal{C}, \mathcal{O}}^{\rm high}$ defined from a partition $\Pi_A \setminus \Pi^\ast_A = \mathcal{C} \sqcup \mathcal{O}$; see Section \ref{ss:NO_body_flag} for the precise definition. 
Then the Newton--Okounkov body $\Delta(G/B, \mathcal{L}_\lambda, \nu_{\mathcal{C}, \mathcal{O}}^{\rm high})$ is unimodularly equivalent to the Gelfand--Tsetlin polytope $GT(\lambda)$ for every partition $\Pi_A \setminus \Pi^\ast_A = \mathcal{C} \sqcup \mathcal{O}$ (see \cref{t:highest_term_valuation}). 
The situation is quite different from the case of $\nu_{\mathcal{C}, \mathcal{O}}^{\rm low}$ since the combinatorics of a marked chain-order polytope heavily depends on the choice of a partition $\Pi_A \setminus \Pi^\ast_A = \mathcal{C} \sqcup \mathcal{O}$. 

The present paper is organized as follows. 
In Section 2, we review some basic definitions and facts on flag varieties.  
Section 3 is devoted to recalling some basic facts on Kashiwara crystal bases for fundamental representations, which are used in the proof of our main result.
In Section 4, we recall the definition of marked chain-order polytopes, and review their Minkowski decomposition property proved in \cite{FFP}.
In Section 5, we recall the definition of Newton--Okounkov bodies, and define our main valuations. 
In Section 6, we prove Theorems 1, 2 above, and construct a specific basis of $V(\lambda)$ from the set of lattice points in a marked chain-order polytope.
We also discuss $\nu_{\mathcal{C}, \mathcal{O}}^{\rm high}$ in this section. 
Section 7 is devoted to studying the case of $Sp_{2n}(\c)$. 
Through a concrete example, we show that Theorem 1 above cannot be naturally extended to this case. 

\begin{ack}\normalfont
The author is grateful to Yunhyung Cho, Akihiro Higashitani, and Eunjeong Lee for helpful comments and fruitful discussions. 
\end{ack}

\section{Basic definitions on flag varieties}\label{s:flag_variety}

Let $G \coloneqq SL_{n+1}(\c)$ be the special linear group (of type $A_n$), and $\mathfrak{g} \coloneqq \mathfrak{sl}_{n+1}(\c)$ the Lie algebra of $G$. 
We identify the set $I$ of vertices of the Dynkin diagram with $\{1, 2, \ldots, n\}$ as follows:
\begin{align*}
&A_n\ \begin{xy}
\ar@{-} (50,0) *++!D{1} *\cir<3pt>{};
(60,0) *++!D{2} *\cir<3pt>{}="C"
\ar@{-} "C";(65,0) \ar@{.} (65,0);(70,0)^*!U{}
\ar@{-} (70,0);(75,0) *++!D{n-1} *\cir<3pt>{}="D"
\ar@{-} "D";(85,0) *++!D{n} *\cir<3pt>{}="E"
\end{xy}.
\end{align*}
Let us take a Borel subgroup $B \subseteq G$ to be the subgroup of upper triangular matrices. 
The quotient space $G/B$ is called the \emph{full flag variety}, which is a nonsingular projective variety. 
We denote by $H \subseteq B$ the subgroup of diagonal matrices, and by $\mathfrak{h} \subseteq \mathfrak{g}$ the Lie algebra of $H$.
Let $\mathfrak{h}^\ast \coloneqq {\rm Hom}_\c (\mathfrak{h}, \c)$ be the dual space of $\mathfrak{h}$, $\langle \cdot, \cdot \rangle \colon \mathfrak{h}^\ast \times \mathfrak{h} \rightarrow \c$ the canonical pairing, $P \subseteq \mathfrak{h}^\ast$ the weight lattice for $\mathfrak{g}$, and $P_+ \subseteq P$ the set of dominant integral weights.
For $\lambda \in P$, there uniquely exists a character $\tilde{\lambda} \colon H \rightarrow \c^\times$ of $H$ such that $d \tilde{\lambda} = \lambda$, where $\c^\times \coloneqq \c \setminus \{0\}$.
By composing this with the canonical projection $B \twoheadrightarrow H$, we obtain a character $\tilde{\lambda} \colon B \rightarrow \c^\times$ of $B$, which is also denoted by $\tilde{\lambda}$. 
For $\lambda \in P$, define a line bundle $\mathcal{L}_\lambda$ on $G/B$ as follows:
\begin{equation}\label{eq:line_bundle_type_A}
\begin{aligned}
\mathcal{L}_\lambda \coloneqq (G \times \mathbb{C})/B,
\end{aligned}
\end{equation}
where the right $B$-action on $G \times \mathbb{C}$ is given by $(g, c) \cdot b \coloneqq (g b, \tilde{\lambda}(b) c)$ for $g \in G$, $c \in \mathbb{C}$, and $b \in B$. 
If $\lambda \in P_+$, then the line bundle $\mathcal{L}_\lambda$ on $G/B$ is generated by global sections (see \cite[Proposition 1.4.1]{Bri}). 
For $\lambda \in P_+$, we denote by $V(\lambda)$ the irreducible highest weight $G$-module over $\c$ with highest weight $\lambda$ and with highest weight vector $v_\lambda$. 
Define a morphism $\rho_\lambda \colon G/B \rightarrow \mathbb{P}(V(\lambda))$ by
\[\rho_\lambda (g \bmod B) \coloneqq \c g v_\lambda.\] 
Then we have $\rho_\lambda ^\ast (\mathcal{O}(1)) = \mathcal{L}_\lambda$. 
Hence the morphism $\rho_\lambda$ induces a $\c$-linear map 
\[\rho_\lambda ^\ast \colon H^0(\mathbb{P}(V(\lambda)), \mathcal{O}(1)) \rightarrow H^0(G/B, \mathcal{L}_\lambda).\] 
Note that the space $H^0(\mathbb{P}(V(\lambda)), \mathcal{O}(1))$ of global sections is naturally identified with the dual $G$-module $V(\lambda)^\ast \coloneqq {\rm Hom}_\c (V(\lambda), \c)$. 
By the Borel--Weil theorem (see, for instance, \cite[Section 8.1.21 and Corollary 8.1.26]{Kum}), we know that the $\c$-linear map $\rho_\lambda ^\ast$ gives an isomorphism of $G$-modules from $V(\lambda)^\ast$ to $H^0(G/B, \mathcal{L}_\lambda)$.  
Hence, for $\sigma, \tau \in H^0(G/B, \mathcal{L}_\lambda) \setminus \{0\}$, the rational function $\sigma/\tau \in \c(G/B) \setminus \{0\}$ is given by
\begin{equation}\label{eq:rational_map_computation}
\begin{aligned}
(\sigma/\tau)(g \bmod B) = \sigma(g v_\lambda)/\tau(g v_\lambda)
\end{aligned}
\end{equation}
for $g \in G$ such that $\tau(g v_\lambda) \neq 0$. 
Denoting the normalizer of $H$ in $G$ by $N_G(H)$, the \emph{Weyl group} $W$ is defined to be the quotient group $N_G(H)/H$.
This group is generated by the set $\{s_i \mid i \in I\}$ of simple reflections. 
A sequence ${\bm i} = (i_1, \ldots, i_m) \in I^m$ is called a \emph{reduced word} for $w \in W$ if $w = s_{i_1} \cdots s_{i_m}$ and if $m$ is the minimum among such expressions of $w$. 
In this case, $m$ is called the \emph{length} of $w$, which is denoted by $\ell(w)$.

\begin{defi}[{see, for instance, \cite[Section I\hspace{-.1em}I.13.3]{Jan} and \cite[Definition 7.1.13]{Kum}}]\normalfont
For $w \in W$, we define a closed subvariety $X(w)$ of $G/B$ to be the Zariski closure of $B \widetilde{w} B/B$ in $G/B$, where $\widetilde{w} \in N_G(H)$ is a lift for $w \in W = N_G(H)/H$. 
This variety $X(w)$ is called a \emph{Schubert variety}.
\end{defi} 

For $w \in W$, the Schubert variety $X(w)$ is a normal projective variety, and we have $\dim_\c (X(w)) = \ell(w)$ (see, for instance, \cite[Sections I\hspace{-.1em}I.13.3, I\hspace{-.1em}I.14.15]{Jan}). 
Denoting the longest element in $W$ by $w_0$, we see that the Schubert variety $X(w_0)$ coincides with $G/B$. 
For $1 \leq i, j \leq n+1$, let $E_{i, j}$ denote the $(n+1) \times (n+1)$-matrix whose $(i, j)$-entry is $1$ and other entries are all $0$. 
Then we can take Chevalley generators $e_i, f_i, h_i \in \mathfrak{g}$, $i \in I$, as $e_i \coloneqq E_{i, i+1}$, $f_i \coloneqq E_{i+1, i}$, and $h_i \coloneqq E_{i, i} - E_{i+1, i+1}$. 
For $i \in I$, set
\begin{align*}
\overline{s}_i &\coloneqq \exp(f_i) \exp(-e_i) \exp(f_i)\\
&= E_{i+1, i} - E_{i, i+1} + \sum_{1 \leq j \leq n+1; j \neq i, i+1} E_{j, j} \in N_G(H),
\end{align*}
which is a lift for $s_i \in W$. 
For $w \in W = N_G(H)/H$, we define its lift $\overline{w} \in N_G(H)$ by 
\begin{equation}\label{eq:lifts_for_Weyl_group_elements}
\begin{aligned}
\overline{w} \coloneqq \overline{s}_{i_1} \overline{s}_{i_2} \cdots \overline{s}_{i_m},
\end{aligned}
\end{equation}
where $(i_1, i_2, \ldots, i_m)$ is a reduced word for $w$. 
The element $\overline{w}$ does not depend on the choice of a reduced word $(i_1, i_2, \ldots, i_m)$ for $w$. 
If $w$ is the identity element $e$, then the lift $\overline{e} \in N_G(H)$ is defined to be the identity matrix.
Denote by $\tau_\lambda \in H^0(G/B, \mathcal{L}_\lambda) \simeq V(\lambda)^\ast$ the unique lowest weight vector satisfying the condition that $\langle \tau_\lambda, v_\lambda \rangle = 1$.
Let $\{\varpi_i \mid i \in I\} \subseteq  P_+$ be the set of fundamental weights. 
For $k \in I$, the representations $V(\varpi_k)$ and $H^0(G/B, \mathcal{L}_{\varpi_k})$ are minuscule (see \cite[Section 13.4]{LR}), that is, we have
\begin{align*}
V(\varpi_k) = \sum_{w \in W} \c \overline{w} v_{\varpi_k}\quad \text{and}\quad H^0(G/B, \mathcal{L}_{\varpi_k}) = \sum_{w \in W} \c \overline{w} \tau_{\varpi_k}.
\end{align*}

\section{Crystal bases for fundamental representations}\label{s:crystal_basis}

In this section, we recall some basic facts on the Kashiwara crystal basis for a fundamental representation, which are used in Section \ref{ss:proof}.
See Kashiwara's survey \cite{Kas} for more details. 
For $k \in I$, the crystal basis $\mathcal{B}(\varpi_k)$ is a specific combinatorial skeleton of $V(\varpi_k)$, which is equipped with specific operators $\{\tilde{e}_i \mid i \in I\} \cup \{\tilde{f}_i \mid i \in I\}$, called \emph{Kashiwara operators}.
An explicit description of $\mathcal{B}(\varpi_k)$ using Young tableaux is given in \cite[Section 3.3]{KasNak}. 
More precisely, the crystal basis $\mathcal{B}(\varpi_k)$ is identified with 
\[\{(j_1, j_2, \ldots, j_k) \mid 1 \leq j_1 < j_2 < \cdots < j_k \leq n+1\}\]
as a set, where $(j_1, j_2, \ldots, j_k)$ corresponds to the Young tableau with only one column of $k$ boxes whose entries are $j_1, j_2, \ldots, j_k$. 
In addition, the actions of Kashiwara operators are given by
\begin{align*}
&\tilde{e}_i (j_1, j_2, \ldots, j_k) = 
\begin{cases} 
(j_1, \ldots, j_{\ell-1}, j_\ell-1, j_{\ell+1}, \ldots, j_k) &\text{if}\ j_\ell = i+1\ \text{for\ some}\ \ell\ \text{and}\ j_{\ell-1} \neq i,\\ 
0 &\text{otherwise}, 
\end{cases}\\
&\tilde{f}_i (j_1, j_2, \ldots, j_k) = 
\begin{cases} 
(j_1, \ldots, j_{\ell-1}, j_\ell+1, j_{\ell+1}, \ldots, j_k) &\text{if}\ j_\ell = i\ \text{for\ some}\ \ell\ \text{and}\ j_{\ell+1} \neq i+1,\\ 
0 &\text{otherwise}
\end{cases}
\end{align*}
for $1 \leq i \leq n$ and $(j_1, j_2, \ldots, j_k) \in \mathcal{B}(\varpi_k)$, where $0$ is an additional element which is not contained in $\mathcal{B}(\varpi_k)$, and we set $j_0 \coloneqq 0$, $j_{k+1} \coloneqq n+2$.
If we set 
\begin{align*}
\varepsilon_i (b) \coloneqq \max\{a \in \z_{\geq 0} \mid \tilde{e}_i^a b \neq 0\}\quad \text{and}\quad \varphi_i (b) \coloneqq \max\{a \in \z_{\geq 0} \mid \tilde{f}_i^a b \neq 0\}
\end{align*}
for $1 \leq i \leq n$ and $b \in \mathcal{B}(\varpi_k)$, then it follows that 
\begin{align*}
&\varepsilon_i (b) = 
\begin{cases} 
1 &\text{if}\ \tilde{e}_i b \neq 0,\\ 
0 &\text{otherwise}, 
\end{cases}\\
&\varphi_i (b) = 
\begin{cases} 
1 &\text{if}\ \tilde{f}_i b \neq 0,\\ 
0 &\text{otherwise}.
\end{cases}
\end{align*}

\begin{ex}
The crystal basis $\mathcal{B}(\varpi_1)$ is realized as follows: 
\begin{align*}
&(1) \xrightarrow{\tilde{f}_1} (2) \xrightarrow{\tilde{f}_2} \cdots \xrightarrow{\tilde{f}_{n-1}} (n) \xrightarrow{\tilde{f}_n} (n+1).
\end{align*}
\end{ex}

\begin{ex}
Let $G = SL_4(\c)$. 
Then the crystal basis $\mathcal{B}(\varpi_2)$ is realized as follows: 
\[\xymatrix{
 & & {(2, 3)} \ar[dr]^-{\tilde{f}_3} & & \\
{(1, 2)} \ar[r]^-{\tilde{f}_2} & {(1, 3)} \ar[ur]^-{\tilde{f}_1} \ar[dr]_-{\tilde{f}_3} & & {(2, 4)} \ar[r]^-{\tilde{f}_2} & {(3, 4).}\\
 & & {(1, 4)} \ar[ur]_-{\tilde{f}_1} & &
}\]
\end{ex}

Let $\{G^{\rm low}_{\varpi_k} (b) \mid b \in \mathcal{B}(\varpi_k)\} \subseteq V(\varpi_k)$ denote (the specialization at $q = 1$ of) the lower global basis (see \cite[Section 12.3]{Kas}), and $\{G^{\rm up}_{\varpi_k} (b) \mid b \in \mathcal{B}(\varpi_k)\} \subseteq H^0(G/B, \mathcal{L}_{\varpi_k}) \simeq V(\varpi_k)^\ast$ its dual basis. 
Then we have
\begin{align*}
e_i G^{\rm low}_{\varpi_k} (b) &= 
\begin{cases} 
c_1 G^{\rm low}_{\varpi_k} (\tilde{e}_i b) &\text{if}\ \tilde{e}_i b \neq 0,\\ 
0 &\text{otherwise},
\end{cases}\\
f_i G^{\rm low}_{\varpi_k} (b) &= 
\begin{cases} 
c_2 G^{\rm low}_{\varpi_k} (\tilde{f}_i b) &\text{if}\ \tilde{f}_i b \neq 0,\\ 
0 &\text{otherwise}
\end{cases}
\end{align*}
for some $c_1, c_2 \in \c^\times$ (see \cite[Section 12.4]{Kas}). 
If we write 
\[{\rm wt}(b) \coloneqq \sum_{1 \leq i \leq n} (\varphi_i (b) - \varepsilon_i (b)) \varpi_i\]
for $b \in \mathcal{B}(\varpi_k)$, then ${\rm wt}(b)$ coincides with the weight of $G^{\rm low}_{\varpi_k} (b)$ (see \cite[equation (4.3)]{Kas}). 
Let $b_{\varpi_k} \in \mathcal{B}(\varpi_k)$ denote the element corresponding to $(1, 2, \ldots, k)$. 
Then we have ${\rm wt}(b_{\varpi_k}) = \varpi_k$, which implies that $G^{\rm low}_{\varpi_k} (b_{\varpi_k}) \in \c^\times v_{\varpi_k}$. 
For $1 \leq i \leq n$, define $s_i \colon \mathcal{B}(\varpi_k) \rightarrow \mathcal{B}(\varpi_k)$ by
\[s_i b \coloneqq 
\begin{cases} 
\tilde{f}_i^{\langle {\rm wt}(b), h_i \rangle} b &\text{if}\ \langle {\rm wt}(b), h_i \rangle \geq 0,\\ 
\tilde{e}_i^{-\langle {\rm wt}(b), h_i \rangle} b &\text{otherwise}.
\end{cases}\]
Then we have $s_i^2 b = b$ and ${\rm wt}(s_i b) = s_i {\rm wt}(b)$ for all $b \in \mathcal{B}(\varpi_k)$ (see \cite[equation (11.2)]{Kas}).
Note that the Weyl group $W$ is isomorphic to the symmetric group $\mathfrak{S}_{n+1}$ by identifying the simple reflection $s_i$ with the transposition $(i\ i+1)$. 
Then $s_i (j_1, j_2, \ldots, j_k)$ is given by rearranging $s_i (j_1), s_i (j_2), \ldots, s_i (j_k)$ in ascending order for all $(j_1, j_2, \ldots, j_k) \in \mathcal{B}(\varpi_k)$. 
Since the representation $V(\varpi_k)$ is minuscule, it holds for $b \in \mathcal{B}(\varpi_k)$ that 
\[\overline{s}_i G^{\rm low}_{\varpi_k} (b) \in \c^\times G^{\rm low}_{\varpi_k} (s_i b).\]

\section{Marked chain-order polytopes}\label{s:marked poset}

In this section, we review some basic definitions and properties of marked chain-order polytopes, following \cite{FFLP, FFP}. 
Let $\Pi$ be a finite poset (i.e.\ a finite partially ordered set) equipped with a partial order $\preceq$, and $\Pi^\ast \subseteq \Pi$ a subset of $\Pi$ containing all minimal and maximal elements in $\Pi$. 
Take an element $\lambda = (\lambda_a)_{a \in \Pi^\ast} \in \r^{\Pi^\ast}$, called a \emph{marking}, such that $\lambda_a \leq \lambda_b$ if $a \preceq b$ in $\Pi$. 
The triple $(\Pi, \Pi^\ast, \lambda)$ is called a \emph{marked poset}.

\begin{defi}[{\cite[Section 1.3]{FFLP}}]
Fix a partition $\Pi \setminus \Pi^\ast = \mathcal{C} \sqcup \mathcal{O}$. 
Then the \emph{marked chain-order polytope} $\Delta_{\mathcal{C}, \mathcal{O}} (\Pi, \Pi^\ast, \lambda)$ is defined as follows: 
\begin{align*}
\Delta_{\mathcal{C}, \mathcal{O}} (\Pi, \Pi^\ast, \lambda) \coloneqq \{&(x_p)_{p \in \Pi \setminus \Pi^\ast} \in \r^{\Pi \setminus \Pi^\ast} \mid \ x_p \geq 0 \text{ for all}\ p \in \mathcal{C}, \\
&\sum_{i = 1}^\ell x_{p_i} \leq y_b - y_a \text{ for }a \prec p_1 \prec \cdots \prec p_\ell \prec b \text{ with }p_i \in \mathcal{C} \text{ and}\ a,b \in \Pi^\ast \sqcup \mathcal{O}\}, 
\end{align*}
where for $c \in \Pi^\ast \sqcup \mathcal{O}$, we set 
\[y_c \coloneqq 
\begin{cases} 
\lambda_c &\text{if}\ c \in \Pi^\ast,\\ 
x_c &\text{if}\ c \in \mathcal{O}. 
\end{cases}\]
\end{defi}

If $\lambda \in \z^{\Pi^\ast}$, then we see by \cite[Proposition 2.4]{FFLP} that $\Delta_{\mathcal{C}, \mathcal{O}}(\Pi, \Pi^\ast, \lambda)$ is an integral convex polytope for each partition $\Pi \setminus \Pi^\ast = \mathcal{C} \sqcup \mathcal{O}$. 
By definition, the marked chain-order polytope $\Delta_{\emptyset, \Pi \setminus \Pi^\ast} (\Pi, \Pi^\ast, \lambda)$ (resp., $\Delta_{\Pi \setminus \Pi^\ast, \emptyset} (\Pi, \Pi^\ast, \lambda)$) coincides with the marked order polytope $\mathcal{O}(\Pi, \Pi^\ast, \lambda)$ (resp., the marked chain polytope ${\mathcal C}(\Pi, \Pi^\ast, \lambda)$) introduced in \cite[Definition 1.2]{ABS}. 
Generalizing Stanley's transfer map \cite{Sta} for ordinary poset polytopes to marked poset polytopes, Ardila--Bliem--Salazar \cite[Theorem 3.4]{ABS} gave a \emph{transfer map} $\phi \colon \mathcal{O}(\Pi, \Pi^\ast, \lambda) \rightarrow {\mathcal C}(\Pi, \Pi^\ast, \lambda)$, which is an explicit bijective piecewise-affine map. 
Fang--Fourier--Litza--Pegel \cite{FFLP} constructed analogous transfer maps for marked chain-order polytopes. 
More precisely, they defined a piecewise-affine map $\phi_{\mathcal{C}, \mathcal{O}} \colon \r^{\Pi \setminus \Pi^\ast} \rightarrow \r^{\Pi \setminus \Pi^\ast}$, $(x_p)_p \mapsto (x_p^\prime)_p$, by 
\begin{align*}
&x_p^\prime \coloneqq 
\begin{cases}
\min(\{x_p-x_{p'} \mid p' \lessdot p, \ p' \in \Pi \setminus \Pi^\ast\} \cup \{x_p-\lambda_{p'} \mid p' \lessdot p, \ p' \in \Pi^\ast\}) &\text{if}\ p \in \mathcal{C}, \\
x_p &\text{otherwise}
\end{cases}
\end{align*}
for $p \in \Pi \setminus \Pi^\ast$, where for $p, q \in \Pi$, $q \lessdot p$ means that $p$ \emph{covers} $q$, that is, $q \prec p$ and there is no $q' \in \Pi \setminus \{p, q\}$ with $q \prec q' \prec p$. 
Note that $\phi_{\emptyset, \Pi \setminus \Pi^\ast}$ is the identity map, and that $\phi_{\Pi \setminus \Pi^\ast, \emptyset}$ is precisely Ardila--Bliem--Salazar's transfer map $\phi$.

\begin{thm}[{see \cite[Theorem 2.1 and Corollary 2.5]{FFLP}}]\label{t:marked_chain-order_Ehrhart}
For each partition $\Pi \setminus \Pi^\ast = \mathcal{C} \sqcup \mathcal{O}$, the piecewise-affine map $\phi_{\mathcal{C}, \mathcal{O}}$ gives a bijective map from $\mathcal{O}(\Pi, \Pi^\ast, \lambda)$ to $\Delta_{\mathcal{C}, \mathcal{O}}(\Pi, \Pi^\ast, \lambda)$, which induces a bijection between the sets of lattice points. 
As a consequence, the Ehrhart polynomial of $\Delta_{\mathcal{C}, \mathcal{O}}(\Pi, \Pi^\ast, \lambda)$ coincides with that of $\mathcal{O}(\Pi, \Pi^\ast, \lambda)$.
\end{thm}

\begin{rem}
Under some condition on $(\Pi, \Pi^\ast, \lambda)$, the author and Higashitani \cite[Theorem 5.3]{FH} proved that the transfer map $\phi \colon \mathcal{O}(\Pi, \Pi^\ast, \lambda) \rightarrow {\mathcal C}(\Pi, \Pi^\ast, \lambda)$ can be described as a dual operation of a combinatorial mutation up to unimodular equivalence. 
This result is naturally extended to the map $\phi_{\mathcal{C}, \mathcal{O}} \colon \r^{\Pi \setminus \Pi^\ast} \rightarrow \r^{\Pi \setminus \Pi^\ast}$ under the same condition on $(\Pi, \Pi^\ast, \lambda)$. 
\end{rem}

For $\lambda \in P_+$ and $1 \leq k \leq n$, we write $\lambda_{\geq k} \coloneqq \sum_{k \leq \ell \leq n} \langle \lambda, h_\ell \rangle$. 
In this paper, we restrict ourselves to the \emph{Gelfand--Tsetlin poset} $(\Pi_A, \Pi^\ast_A, \lambda)$ of type $A_n$ whose marked Hasse diagram is given in Figure \ref{type_A_marked_Hasse}, where the circles (resp., the rectangles) denote the elements of $\Pi_A \setminus \Pi^\ast_A$ (resp., $\Pi^\ast_A$), and we write 
\[\Pi_A \setminus \Pi^\ast_A = \{q_j^{(i)} \mid 1 \leq i \leq n,\ 1 \leq j \leq n + 1 - i\}.\]
Note that the marking $(\lambda_a)_{a \in \Pi^\ast_A}$ is given as $(0, \lambda_{\geq n}, \ldots, \lambda_{\geq 2}, \lambda_{\geq 1})$, which is also denoted by $\lambda$.
\begin{figure}[!ht]
\begin{center}
   \includegraphics[width=10.0cm,bb=40mm 120mm 170mm 230mm,clip]{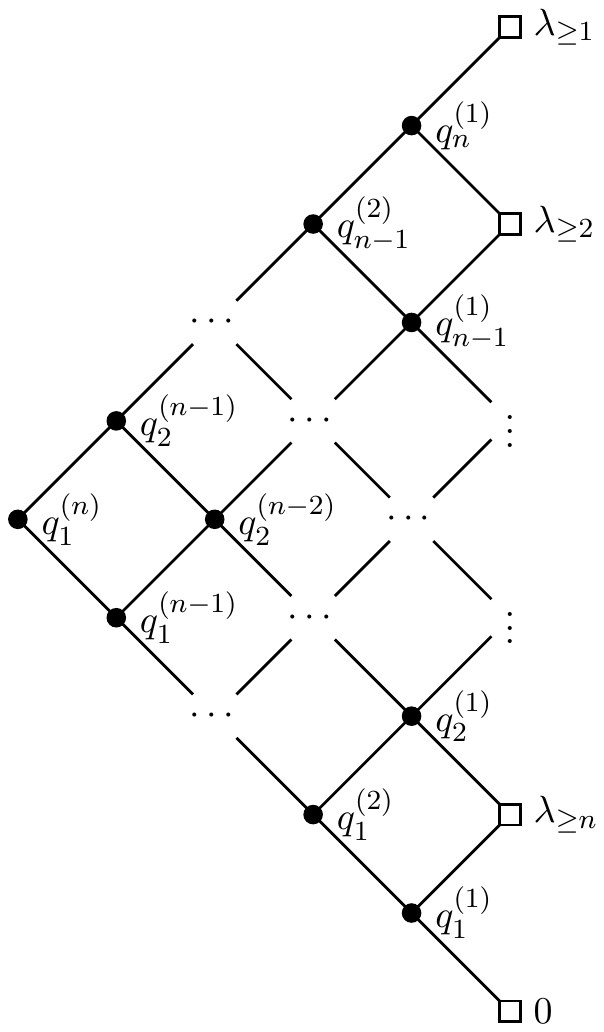}
	\caption{The marked Hasse diagram of the Gelfand--Tsetlin poset $(\Pi_A, \Pi^\ast_A, \lambda)$ of type $A_n$.}
	\label{type_A_marked_Hasse}
\end{center}
\end{figure}
By definition, the marked order polytope $\mathcal{O}(\Pi_A, \Pi^\ast_A, \lambda)$ coincides with the Gelfand--Tsetlin polytope $GT(\lambda)$ (see \cite[Section 5]{Lit} for the definition), and the marked chain polytope $\mathcal{C}(\Pi_A, \Pi^\ast_A, \lambda)$ coincides with the FFLV polytope $FFLV(\lambda)$ (see \cite[equation (0.1)]{FeFL1} for the definition). 
Let $\Delta_{\mathcal{C}, \mathcal{O}}^\z(\Pi_A, \Pi^\ast_A, \lambda)$ denote the set of lattice points in $\Delta_{\mathcal{C}, \mathcal{O}}(\Pi_A, \Pi^\ast_A, \lambda)$. 
Then we see by \cref{t:marked_chain-order_Ehrhart} that 
\begin{equation}\label{eq:marked_chain_order_number_of_lattice_points}
\begin{aligned}
|\Delta_{\mathcal{C}, \mathcal{O}}^\z(\Pi_A, \Pi^\ast_A, \lambda)| &= |GT(\lambda) \cap \z^{\Pi_A \setminus \Pi_A^\ast}|\\
&= \dim_\c (V(\lambda)).
\end{aligned}
\end{equation}
For subsets $X, Y \subseteq \r^{\Pi \setminus \Pi^\ast}$, define $X+Y \subseteq \r^{\Pi \setminus \Pi^\ast}$ to be the Minkowski sum:
\[X+Y \coloneqq \{x+y \mid x \in X,\ y \in Y\}.\]
The marked chain-order polytope $\Delta_{\mathcal{C}, \mathcal{O}}(\Pi_A, \Pi^\ast_A, \lambda)$ decomposes as the Minkowski sum of those for fundamental weights as follows. 

\begin{thm}[{see \cite[Theorem 18]{FFP}}]\label{t:decomposition_property_marked_poset}
For each partition $\Pi_A \setminus \Pi^\ast_A = \mathcal{C} \sqcup \mathcal{O}$ and $\lambda = \lambda_1 \varpi_1 + \cdots + \lambda_n \varpi_n \in P_+$, the following equalities hold:
\begin{align*}
&\Delta_{\mathcal{C}, \mathcal{O}}(\Pi_A, \Pi^\ast_A, \lambda) = \lambda_1 \Delta_{\mathcal{C}, \mathcal{O}}(\Pi_A, \Pi^\ast_A, \varpi_1) + \cdots + \lambda_n \Delta_{\mathcal{C}, \mathcal{O}}(\Pi_A, \Pi^\ast_A, \varpi_n),\\
&\Delta_{\mathcal{C}, \mathcal{O}}^\z(\Pi_A, \Pi^\ast_A, \lambda) = \lambda_1 \Delta_{\mathcal{C}, \mathcal{O}}^\z(\Pi_A, \Pi^\ast_A, \varpi_1) + \cdots + \lambda_n \Delta_{\mathcal{C}, \mathcal{O}}^\z(\Pi_A, \Pi^\ast_A, \varpi_n).
\end{align*}
\end{thm}

\begin{ex}\label{ex:marked_chain-order}
Let $n = 2$, and $a_1^{(1)}, a_1^{(2)}, a_2^{(1)}$ denote the coordinate functions for $\Delta_{\mathcal{C}, \mathcal{O}}(\Pi_A, \Pi_A^\ast, \lambda) \subseteq \r^{\Pi_A \setminus \Pi_A^\ast} \simeq \r^3$ corresponding to $q_1^{(1)}, q_1^{(2)}, q_2^{(1)}$, respectively.
Then the marked order polytope $\mathcal{O}(\Pi_A, \Pi_A^\ast, \lambda)$ is given by the following inequalities: 
\[0 \leq a_1^{(1)} \leq a_1^{(2)} \leq a_2^{(1)} \leq \lambda_{\geq 1},\quad a_1^{(1)} \leq \lambda_{\geq 2} \leq a_2^{(1)},\]
and the marked chain polytope ${\mathcal C}(\Pi_A, \Pi_A^\ast, \lambda)$ is the set of $(a_1^{(1)}, a_1^{(2)}, a_2^{(1)}) \in \r_{\geq 0}^3$ satisfying the following inequalities: 
\[a_1^{(1)} \leq \lambda_{\geq 2},\quad a_2^{(1)} \leq \lambda_{\geq 1} - \lambda_{\geq 2},\quad a_1^{(1)} + a_1^{(2)} + a_2^{(1)} \leq \lambda_{\geq 1}.\]
In addition, the transfer map $\phi \colon \mathcal{O}(\Pi_A, \Pi_A^\ast, \lambda) \rightarrow {\mathcal C}(\Pi_A, \Pi_A^\ast, \lambda)$ is given by
\[(a_1^{(1)}, a_1^{(2)}, a_2^{(1)}) \mapsto (a_1^{(1)}, a_1^{(2)} - a_1^{(1)}, \min\{a_2^{(1)}-a_1^{(2)}, a_2^{(1)} - \lambda_{\geq 2}\}).\]
Let us define a partition $\Pi_A \setminus \Pi_A^\ast = \mathcal{C} \sqcup \mathcal{O}$ by $\mathcal{C} = \{q_1^{(2)}\}$ and $\mathcal{O} = \{q_1^{(1)}, q_2^{(1)}\}$. 
Then the marked chain-order polytope $\Delta_{\mathcal{C}, \mathcal{O}} (\Pi_A, \Pi_A^\ast, \lambda)$ is the set of $(a_1^{(1)}, a_1^{(2)}, a_2^{(1)}) \in \r^3$ satisfying the following inequalities: 
\[0 \leq a_1^{(1)} \leq \lambda_{\geq 2},\quad \lambda_{\geq 2} \leq a_2^{(1)} \leq \lambda_{\geq 1},\quad 0 \leq a_1^{(2)} \leq a_2^{(1)} - a_1^{(1)}.\]
In addition, the transfer map $\phi_{\mathcal{C}, \mathcal{O}} \colon \mathcal{O}(\Pi_A, \Pi_A^\ast, \lambda) \rightarrow \Delta_{\mathcal{C}, \mathcal{O}}(\Pi_A, \Pi_A^\ast, \lambda)$ is given by 
\[(a_1^{(1)}, a_1^{(2)}, a_2^{(1)}) \mapsto (a_1^{(1)}, a_1^{(2)} - a_1^{(1)}, a_2^{(1)}).\]
\end{ex}

\section{Newton--Okounkov bodies}\label{s:NO body}

In Section \ref{ss:NO_body}, we recall the definitions of higher rank valuations and Newton--Okounkov bodies, following \cite{Kav, KK2}.
In Section \ref{ss:NO_body_flag}, we define our main valuations and review some previous realizations of Gelfand--Tsetlin polytopes \cite{Oko1, FO} and FFLV polytopes \cite{Kir1} as Newton--Okounkov bodies.

\subsection{Basic definitions on Newton--Okounkov bodies}\label{ss:NO_body}

Let $R$ be a $\mathbb{C}$-algebra without nonzero zero-divisors. 
We fix $N \in \z_{>0}$ and a total order $\leq$ on $\z^N$ respecting the addition. 

\begin{defi}\label{d:valuation}\normalfont
A map $\nu \colon R \setminus \{0\} \rightarrow \mathbb{Z}^N$ is called a \emph{valuation} on $R$ with values in $\z^N$ if for each $\sigma, \tau \in R \setminus \{0\}$ and $c \in \c^\times$, it holds that
\begin{enumerate}
\item[{\rm (i)}] $\nu(\sigma \cdot \tau) = \nu(\sigma) + \nu(\tau)$,
\item[{\rm (ii)}] $\nu(c \cdot \sigma) = \nu(\sigma)$, 
\item[{\rm (iii)}] $\nu(\sigma + \tau) \geq {\rm min} \{\nu(\sigma), \nu(\tau)\}$ unless $\sigma + \tau = 0$. 
\end{enumerate}
\end{defi}

For ${\bm a} \in \z^N$ and a valuation $\nu$ on $R$ with values in $\z^N$, define a $\c$-subspace $R_{\bm a} \subseteq R$ by
\[R_{\bm a} \coloneqq \{\sigma \in R \setminus \{0\} \mid \nu(\sigma) \geq {\bm a}\} \cup \{0\}.\]
Then we set $\widehat{R}_{\bm a} \coloneqq R_{\bm a}/\bigcup_{{\bm a} < {\bm b}} R_{\bm b}$, which is called the \emph{leaf} above ${\bm a} \in \z^N$. 
A valuation $\nu$ is said to \emph{have $1$-dimensional leaves} if $\dim_\c(\widehat{R}_{\bm a}) = 0\ {\rm or}\ 1$ for all ${\bm a} \in \z^N$. 

\begin{ex}\label{ex:lowest_and_highest_term_valuation}
Let $\mathbb{C}(z_1, \ldots, z_N)$ be the field of rational functions in $N$ variables, and take a total order $\leq$ on $\mathbb{Z}^N$ to be the lexicographic order, that is, $(a_1, \ldots, a_N) < (a_1 ^\prime, \ldots, a_N ^\prime)$ if and only if there exists $1 \leq k \leq N$ such that 
\begin{align*}
a_1 = a_1 ^\prime, \ldots, a_{k-1} = a_{k-1} ^\prime,\ a_k < a_k ^\prime.
\end{align*}
The lexicographic order $\leq$ on $\mathbb{Z}^N$ gives a total order (denoted by the same symbol $\leq$) on the set of Laurent monomials in $z_1, \ldots, z_N$ by
\begin{center}
$z_1 ^{a_1} \cdots z_N ^{a_N} \leq z_1 ^{a_1 ^\prime} \cdots z_N ^{a_N ^\prime}$ if and only if $(a_1, \ldots, a_N) \leq (a_1 ^\prime, \ldots, a_N ^\prime)$. 
\end{center}
We define two maps $\nu^{\rm low}_{z_1 > \cdots > z_N}, \nu^{\rm high}_{z_1 > \cdots > z_N} \colon \mathbb{C}(z_1, \ldots, z_N) \setminus \{0\} \rightarrow \z^N$ by 
\begin{itemize}
\item $\nu^{\rm low} _{z_1 > \cdots > z_N} (f) \coloneqq (a_1, \ldots, a_N)$ and $\nu^{\rm high} _{z_1 > \cdots > z_N} (f) \coloneqq -(a_1 ^\prime, \ldots, a_N ^\prime)$ for
\begin{align*}
f &= c z_1 ^{a_1} \cdots z_N ^{a_N} + ({\rm higher\ terms})\\
&= c^\prime z_1 ^{a_1 ^\prime} \cdots z_N ^{a_N ^\prime} + ({\rm lower\ terms}) \in \c[z_1, \ldots, z_N] \setminus \{0\},
\end{align*}
where $c, c^\prime \in \c^\times$, and we mean by ``(higher terms)'' (resp., ``(lower terms)'') a linear combination of monomials bigger than $z_1 ^{a_1} \cdots z_N ^{a_N}$ (resp., smaller than $z_1 ^{a_1^\prime} \cdots z_N ^{a_N^\prime}$) with respect to the total order $\leq$,
\item $\nu^{\rm low}_{z_1 > \cdots > z_N} (f/g) \coloneqq \nu^{\rm low} _{z_1 > \cdots > z_N} (f) - \nu^{\rm low} _{z_1 > \cdots > z_N} (g)$ and $\nu^{\rm high}_{z_1 > \cdots > z_N} (f/g) \coloneqq \nu^{\rm high} _{z_1 > \cdots > z_N} (f) - \nu^{\rm high} _{z_1 > \cdots > z_N} (g)$ for all nonzero polynomials $f, g \in \c[z_1, \ldots, z_N] \setminus \{0\}$.
\end{itemize}
Then the map $\nu^{\rm low} _{z_1 > \cdots > z_N}$ (resp., $\nu^{\rm high} _{z_1 > \cdots > z_N}$) is a valuation with respect to the lexicographic order $\leq$ whose leaves are all $1$-dimensional, which is called the \emph{lowest term valuation} (resp., the \emph{highest term valuation}) with respect to the \emph{lexicographic order} $z_1 > \cdots > z_N$.
\end{ex}

The following is a fundamental property of valuations with $1$-dimensional leaves.

\begin{prop}[{see, for instance, \cite[Proposition 1.9]{Kav}}]\label{p:property_valuation}
Let $\nu \colon R \setminus \{0\} \rightarrow \mathbb{Z}^N$ be a valuation with $1$-dimensional leaves, and $V \subseteq R$ a finite-dimensional $\c$-linear subspace. 
\begin{enumerate}
\item[{\rm (1)}] There exists a $\c$-basis ${\mathbf B}$ of $V$ such that the values $\nu (b)$, $b \in {\mathbf B}$, are all distinct. 
\item[{\rm (2)}] The equality $\nu (V \setminus \{0\}) = \{\nu(b) \mid b \in {\mathbf B}\}$ holds; in particular, the cardinality of $\nu (V \setminus \{0\})$ coincides with $\dim_\c(V)$.
\end{enumerate}
\end{prop}

Let $X$ be an irreducible normal projective variety over $\c$, $\mathcal{L}$ a line bundle on $X$ generated by global sections, and $N \coloneqq \dim_\c (X)$. 

\begin{defi}[{see \cite[Definition 1.10]{KK2}}]\label{d:Newton--Okounkov body}
Take a valuation $\nu \colon \mathbb{C}(X) \setminus \{0\} \rightarrow \z^N$ with $1$-dimensional leaves, and fix a nonzero section $\tau \in H^0 (X, \mathcal{L})$. 
We define a semigroup $S(X, \mathcal{L}, \nu, \tau) \subseteq \z_{>0} \times \z^N$ by 
\[S(X, \mathcal{L}, \nu, \tau) \coloneqq \bigcup_{k \in \z_{>0}} \{(k, \nu(\sigma / \tau^k)) \mid \sigma \in H^0(X, \mathcal{L}^{\otimes k}) \setminus \{0\}\},\] 
and let $C(X, \mathcal{L}, \nu, \tau) \subseteq \r_{\geq 0} \times \r^N$ denote the smallest real closed cone containing $S(X, \mathcal{L}, \nu, \tau)$. 
Define a convex set $\Delta(X, \mathcal{L}, \nu, \tau) \subseteq \r^N$ by 
\[\Delta(X, \mathcal{L}, \nu, \tau) \coloneqq \{{\bm a} \in \r^N \mid (1, {\bm a}) \in C(X, \mathcal{L}, \nu, \tau)\},\] 
which is called the \emph{Newton--Okounkov body} of $(X, \mathcal{L})$ associated with $\nu$ and $\tau$.
In the notation of \cite[Definition 1.10]{KK2}, the Newton--Okounkov body $\Delta(X, \mathcal{L}, \nu, \tau)$ is $\Delta(S, M)$ for $S \coloneqq S(X, \mathcal{L}, \nu, \tau)$ and $M \coloneqq L(S) \cap (\r_{\geq 0} \times \r^N)$, where $L(S) \subseteq \r \times \r^N$ denotes the smallest $\r$-linear subspace containing $S$ (see also \cite[Section 3.2]{KK2} and \cite[Section 1.2]{Kav}).
\end{defi}

\begin{rem}
For another nonzero section $\tau^\prime \in H^0 (X, \mathcal{L})$, we have
\[
S(X, \mathcal{L}, \nu, \tau^\prime) \cap (\{k\} \times \mathbb{Z}^N) = (S(X, \mathcal{L}, \nu, \tau) \cap (\{k\} \times \mathbb{Z}^N)) + (0, k \nu(\tau/\tau^\prime))
\] 
for all $k \in \mathbb{Z}_{>0}$, which implies that
\[\Delta(X, \mathcal{L}, \nu, \tau^\prime) = \Delta(X, \mathcal{L}, \nu, \tau) + \nu(\tau/\tau^\prime).\] 
Hence $\Delta(X, \mathcal{L}, \nu, \tau)$ does not depend on the choice of $\tau$ up to translations. 
This is the reason why we denote it simply by $\Delta(X, \mathcal{L}, \nu)$ in Introduction.
\end{rem} 

Since $S(X, \mathcal{L}, \nu, \tau)$ is a semigroup, the definition of Newton--Okounkov bodies implies that 
\begin{equation}\label{eq:NO_scalar_multiple}
\begin{aligned}
\Delta(X, \mathcal{L}^{\otimes k}, \nu, \tau^k) = k \Delta(X, \mathcal{L}, \nu, \tau)
\end{aligned}
\end{equation}
for all $k \in \z_{> 0}$. 
In addition, it follows by \cite[Theorem 2.30]{KK2} that the Newton--Okounkov body $\Delta(X, \mathcal{L}, \nu, \tau)$ is a convex body, i.e., a compact convex set. 
Since $\mathcal{L}$ is generated by global sections, we obtain a morphism $X \rightarrow \mathbb{P} (H^0(X, \mathcal{L})^\ast)$; the closure of its image is denoted by $Y_{\mathcal{L}}$. 

\begin{thm}[{\cite[Corollary 3.2]{KK2}}]\label{t:NO_dimension_volume}
The real dimension $d \coloneqq \dim_\r (\Delta(X, \mathcal{L}, \nu, \tau))$ equals the complex dimension of $Y_{\mathcal{L}}$. 
In addition, the degree ${\rm deg}(Y_{\mathcal{L}})$ of the closed embedding $Y_{\mathcal{L}} \hookrightarrow \mathbb{P} (H^0(X, \mathcal{L})^\ast)$ coincides with $\frac{1}{d!} \cdot {\rm Vol}_d (\Delta(X, \mathcal{L}, \nu, \tau))$, where ${\rm Vol}_d (\Delta(X, \mathcal{L}, \nu, \tau))$ denotes the $d$-dimensional volume of $\Delta(X, \mathcal{L}, \nu, \tau)$ with respect to the lattice generated by 
\[\bigcup_{k \in \z_{>0}} \{\nu(\sigma / \tau^k) - \nu(\sigma^\prime / \tau^k) \mid \sigma, \sigma^\prime \in H^0(X, \mathcal{L}^{\otimes k}) \setminus \{0\}\} \subseteq \z^N.\]
\end{thm}

By the additivity of $\nu$, we see that 
\begin{equation}\label{eq:super_additivity_NO}
\begin{aligned}
\Delta(X, \mathcal{L}_1, \nu, \tau_1) + \Delta(X, \mathcal{L}_2, \nu, \tau_2) \subseteq \Delta(X, \mathcal{L}_1 \otimes \mathcal{L}_2, \nu, \tau_1 \cdot \tau_2)
\end{aligned}
\end{equation}
for all globally generated line bundles $\mathcal{L}_1, \mathcal{L}_2$ on $X$ and nonzero sections $\tau_1 \in H^0 (X, \mathcal{L}_1)$, $\tau_2 \in H^0 (X, \mathcal{L}_2)$. 
Let 
\[X_\bullet \colon X = X_0 \supsetneq X_1 \supsetneq \cdots \supsetneq X_N\] 
be a sequence of irreducible normal closed subvarieties of $X$ such that $\dim_\c (X_k) = N-k$ for $0 \leq k \leq N$. 
Denote by $\eta_k$ the generic point of $X_k$ for $1 \leq k \leq N$. 
Since $X_k$ is normal for all $0 \leq k \leq N-1$, the stalk $\mathcal{O}_{\eta_{k+1}, X_k}$ of the structure sheaf $\mathcal{O}_{X_k}$ at $\eta_{k+1}$ is a discrete valuation ring with quotient field $\c(X_k)$. 
Let ${\rm ord}_{X_{k+1}} \colon \c(X_k) \setminus \{0\} \rightarrow \z$ denote the corresponding valuation, and take a generator $u_{k+1} \in \c(X_k)$ of the unique maximal ideal of $\mathcal{O}_{\eta_{k+1}, X_k}$.
We define a map 
\[\nu_{X_\bullet} \colon \c(X) \setminus \{0\} \rightarrow \z^N,\quad f \mapsto (a_1, \ldots, a_N),\] 
as follows (see \cite[Section 2.1]{Oko1}). 
The first coordinate $a_1$ is given by $a_1 \coloneqq {\rm ord}_{X_1}(f)$, which implies that $(u_1 ^{-a_1} f)|_{X_1} \in \c(X_1) \setminus \{0\}$. 
Then the second coordinate $a_2$ is given by $a_2 \coloneqq {\rm ord}_{X_2} ((u_1 ^{-a_1} f)|_{X_1})$. 
Continuing in this way, we define all $a_k$. 
This is the definition of $\nu_{X_\bullet}$. 
Then the map $\nu_{X_\bullet}$ is a valuation with respect to the lexicographic order, which has $1$-dimensional leaves.

\begin{rem}
The valuation $\nu_{X_\bullet}$ depends on the choice of $u_1, \ldots, u_N$, but the corresponding Newton--Okounkov body $\Delta(X, \mathcal{L}, \nu_{X_\bullet}, \tau)$ is independent up to unimodular equivalence. 
\end{rem}

\subsection{Newton--Okounkov bodies of flag varieties}\label{ss:NO_body_flag}

In this subsection, we restrict ourselves to Newton--Okounkov bodies of flag varieties, and define our main valuations. 
Set $N \coloneqq \frac{n (n + 1)}{2}$, and define a reduced word ${\bm i}_A = (i_1, \ldots, i_N) \in I^N$ for the longest element $w_0 \in W$ by
\begin{equation}\label{eq:reduced_word}
\begin{aligned}
{\bm i}_A \coloneqq (1, 2, 1, 3, 2, 1, \ldots, n, n-1, \ldots, 1).
\end{aligned}
\end{equation}
We write $w_{\geq N+1} \coloneqq e$, and set $w_{\geq k} \coloneqq s_{i_k} s_{i_{k+1}} \cdots s_{i_N}$ for $1 \leq k \leq N$.
Arrange the elements of $\Pi_A \setminus \Pi^\ast_A$ as 
\begin{equation}\label{eq:arrangement_of_poset}
\begin{aligned}
(q_1, q_2, \ldots, q_N) \coloneqq (q_1^{(1)}, q_1^{(2)}, q_2^{(1)}, q_1^{(3)}, q_2^{(2)}, q_3^{(1)}, \ldots, q_1^{(n)}, q_2^{(n-1)}, \ldots, q_n^{(1)}).
\end{aligned}
\end{equation}
Then a partition $\Pi_A \setminus \Pi^\ast_A = \mathcal{C} \sqcup \mathcal{O}$ is identified with a partition of the set $\{1, 2, \ldots, N\}$; we regard this set $\{1, 2, \ldots, N\}$ as the set of positions of entries in ${\bm i}_A$. 
Using the arrangement \eqref{eq:arrangement_of_poset}, we identify $\r^{\Pi_A \setminus \Pi^\ast_A}$ with $\r^N$.
Take a partition $\Pi_A \setminus \Pi^\ast_A = \mathcal{C} \sqcup \mathcal{O}$.
For $1 \leq k \leq N$, we define $u_k \in W$ by 
\[u_k \coloneqq  
\begin{cases}
s_{i_k} &\text{if}\ q_k \in \mathcal{C},\\
e &\text{if}\ q_k \in \mathcal{O},
\end{cases}\]
and set $u_{\leq k} \coloneqq u_1 u_2 \cdots u_k \in W$.
Then let $X_\bullet^{\mathcal{C}, \mathcal{O}}$ denote the following sequence of irreducible normal closed subvarieties of $G/B$:
\begin{equation}\label{eq:definition_of_partial_translation}
\begin{aligned}
(G/B =)\ X(w_{\geq 1}) \supsetneq \overline{u_{\leq 1}} X(w_{\geq 2}) \supsetneq \overline{u_{\leq 2}} X(w_{\geq 3}) \supsetneq \cdots \supsetneq \overline{u_{\leq N}} X(w_{\geq N+1})\ (= \overline{u_{\leq N}} B/B).
\end{aligned}
\end{equation}
If $\mathcal{C} = \emptyset$ and $\mathcal{O} = \Pi_A \setminus \Pi^\ast_A$, then $X_\bullet^{\emptyset, \Pi_A \setminus \Pi^\ast_A}$ is the following sequence of Schubert varieties: 
\[(G/B =)\ X(w_{\geq 1}) \supsetneq X(w_{\geq 2}) \supsetneq X(w_{\geq 3}) \supsetneq \cdots \supsetneq X(w_{\geq N+1})\ (= B/B).\]
If $\mathcal{C} = \Pi_A \setminus \Pi^\ast_A$ and $\mathcal{O} = \emptyset$, then $X_\bullet^{\Pi_A \setminus \Pi^\ast_A, \emptyset}$ is the following sequence of translated Schubert varieties: 
\begin{equation}\label{eq:definition_of_full_translation}
\begin{aligned}
(G/B =)\ X(w_{\geq 1}) \supsetneq \overline{s}_{i_1} X(w_{\geq 2}) \supsetneq \overline{s}_{i_1} \overline{s}_{i_2} X(w_{\geq 3}) \supsetneq \cdots \supsetneq \overline{w_0} X(w_{\geq N+1})\ (= \overline{w_0} B/B),
\end{aligned}
\end{equation}
which was studied by Kiritchenko \cite{Kir1}. 
For $1 \leq k \leq N$ and $t \in \c$, we define $\Omega_k (t) \in G$ by $\Omega_k (t) \coloneqq \overline{u_k} \exp(t f_{i_k})$, and set 
\[\Omega_{\mathcal{C}, \mathcal{O}} (t_1, \ldots, t_N) \coloneqq \Omega_1 (t_1) \cdot \Omega_2 (t_2) \cdots \Omega_N (t_N) \in G\]
for $(t_1, \ldots, t_N) \in \c^N$.
Then it is easy to see by \cite[Chapter I\hspace{-.1em}I.13]{Jan} that for $1 \leq k \leq N$, we obtain a birational morphism
\[\c^{N-k+1} \rightarrow \overline{u_{\leq k-1}} X(w_{\geq k}),\quad (t_k, t_{k+1}, \ldots, t_N) \mapsto \Omega_k (t_k) \cdot \Omega_{k+1} (t_{k+1}) \cdots \Omega_N (t_N) \bmod B,\]
and the rational function $t_k \in \c(t_k, \ldots, t_N) \simeq \c(\overline{u_{\leq k-1}} X(w_{\geq k}))$ gives a generator of the unique maximal ideal of the stalk $\mathcal{O}_{\eta_k, \overline{u_{\leq k-1}} X(w_{\geq k})}$ of the structure sheaf $\mathcal{O}_{\overline{u_{\leq k-1}} X(w_{\geq k})}$ at $\eta_k$, where $\eta_k$ denotes the generic point of $\overline{u_{\leq k}} X(w_{\geq k+1})$. 
Considering the case $k = 1$, we see that the following morphism is birational:
\[\widehat{\Omega}_{\mathcal{C}, \mathcal{O}} \colon \c^N \rightarrow G/B,\quad (t_1, \ldots, t_N) \mapsto \Omega_{\mathcal{C}, \mathcal{O}} (t_1, \ldots, t_N) \bmod B.\]
Using this birational morphism, we identify the function field $\c(G/B)$ with the field $\c(t_1, \ldots, t_N)$ of rational functions in $t_1, \ldots, t_N$. 

\begin{ex}
Let $G = SL_4(\c)$, and $(\mathcal{C}, \mathcal{O}) = (\{q_5, q_6\}, \{q_1, q_2, q_3, q_4\})$. 
Then the sequence $X_\bullet^{\mathcal{C}, \mathcal{O}}$ is given by 
\[(G/B =)\ X(w_{\geq 1}) \supsetneq X(w_{\geq 2}) \supsetneq \cdots \supsetneq X(w_{\geq 5})\ (= X(s_2 s_1)) \supsetneq \overline{s}_2 X(s_1) \supsetneq \overline{s}_2 \overline{s}_1 X(e)\ (= \overline{s}_2 \overline{s}_1 B/B).\]
In addition, we have 
\begin{align*}
\Omega_{\mathcal{C}, \mathcal{O}} (t_1, \ldots, t_6) &= \exp(t_1 f_1) \exp(t_2 f_2) \exp(t_3 f_1) \exp(t_4 f_3) \overline{s}_2 \exp(t_5 f_2) \overline{s}_1 \exp(t_6 f_1)\\
&= \begin{pmatrix}
-t_6 & -1 & 0 & 0  \\
-t_5 -t_1 t_6 -t_3 t_6 & -t_1 -t_3 & -1 & 0 \\
1 -t_2 t_5 -t_2 t_3 t_6 & -t_2 t_3 & -t_2 & 0 \\
t_4 & 0 & 0 & 1
\end{pmatrix},
\end{align*}
which gives an identification $\c(G/B) \simeq \c(t_1, \ldots, t_6)$.
\end{ex}

Let $\nu_{\mathcal{C}, \mathcal{O}}^{\rm low}$ and $\nu_{\mathcal{C}, \mathcal{O}}^{\rm high}$ denote the valuations on $\c(G/B) \simeq \c(t_1, \ldots, t_N)$ defined to be $\nu^{\rm low} _{t_1 > \cdots > t_N}$ and $\nu^{\rm high} _{t_1 > \cdots > t_N}$ on $\c(t_1, \ldots, t_N)$, respectively (see \cref{ex:lowest_and_highest_term_valuation}). 
The valuation $\nu_{\mathcal{C}, \mathcal{O}}^{\rm low}$ is identical to the valuation $\nu_{X_\bullet^{\mathcal{C}, \mathcal{O}}}$ associated with the sequence $X_\bullet^{\mathcal{C}, \mathcal{O}}$.
Let us consider two extreme cases $(\mathcal{C}, \mathcal{O}) = (\emptyset, \Pi_A \setminus \Pi^\ast_A)$ and $(\mathcal{C}, \mathcal{O}) = (\Pi_A \setminus \Pi^\ast_A, \emptyset)$.
Then we obtain the Gelfand--Tsetlin polytopes $GT(\lambda)$ and the FFLV polytopes $FFLV(\lambda)$ for all $\lambda \in P_+$ as the corresponding Newton--Okounkov bodies as follows. 

\begin{thm}[{see \cite[Section 3.1]{Oko1}, \cite[Proposition 3.29, Corollary 5.3]{FO}, and \cite[Example 3.12]{Fuj2}}]\label{t:NOBY_realization_of_GT}
Let $(\mathcal{C}, \mathcal{O}) = (\emptyset, \Pi_A \setminus \Pi^\ast_A)$, and $\lambda \in P_+$.
Then the Newton--Okounkov body $\Delta(G/B, \mathcal{L}_\lambda, \nu_{\emptyset, \Pi_A \setminus \Pi^\ast_A}^{\rm low}, \tau_\lambda)$ coincides with the translated Gelfand--Tsetlin polytope $GT(\lambda) - {\bm a}^{\rm high}_\lambda$, where
\[{\bm a}^{\rm high}_\lambda \coloneqq (0, \underbrace{0, \lambda_{\geq n}}_{2}, \underbrace{0, \lambda_{\geq n}, \lambda_{\geq n -1}}_{3}, \ldots, \underbrace{0, \lambda_{\geq n}, \lambda_{\geq n -1}, \ldots, \lambda_{\geq 2}}_{n}) \in \z^N.\]
\end{thm}

\begin{rem}
The lowest term valuation $\nu_{\emptyset, \Pi_A \setminus \Pi^\ast_A}^{\rm low}$ in \cref{t:NOBY_realization_of_GT} is defined from the reduced word ${\bm i}_A$. 
By \cite[Corollary 5.3]{FO}, \cref{t:NOBY_realization_of_GT} is naturally extended to an arbitrary reduced word of general Lie type and to the corresponding Schubert variety if we replace the translated Gelfand--Tsetlin polytope $GT(\lambda) - {\bm a}^{\rm high}_\lambda$ with a Nakashima--Zelevinsky polytope. 
\end{rem}

\begin{thm}[{\cite[Theorem 2.1]{Kir1}}]\label{t:NOBY_realization_of_FFLV}
Let $(\mathcal{C}, \mathcal{O}) = (\Pi_A \setminus \Pi^\ast_A, \emptyset)$, and $\lambda \in P_+$.
Then the Newton--Okounkov body $\Delta(G/B, \mathcal{L}_\lambda, \nu_{\Pi_A \setminus \Pi^\ast_A, \emptyset}^{\rm low}, \overline{w_0} \tau_\lambda)$ coincides with the FFLV polytope $FFLV(\lambda)$.
\end{thm}

\begin{rem}
The proof of \cref{t:NOBY_realization_of_FFLV} in \cite[Theorem 2.1]{Kir1} uses some geometric arguments. 
A more combinatorial proof is given in \cite[Section 3]{Kir2}. 
\end{rem}

We write 
\[\overline{w}_{\mathcal{C}, \mathcal{O}} \coloneqq \Omega_{\mathcal{C}, \mathcal{O}} (0, \ldots, 0) = \overline{u_1} \cdot \overline{u_2} \cdots \overline{u_N} \in N_G(H),\]
and set 
\[w_{\mathcal{C}, \mathcal{O}} \coloneqq \overline{w}_{\mathcal{C}, \mathcal{O}} \bmod H = u_1 u_2 \cdots u_N \in W.\]
Note that $\overline{w}_{\mathcal{C}, \mathcal{O}} \neq \overline{w_{\mathcal{C}, \mathcal{O}}}$ in general. 
For $\lambda \in P_+$, denote by $V(\lambda, w_{\mathcal{C}, \mathcal{O}})$ the $\c$-subspace of $V(\lambda)$ spanned by weight vectors whose weights are different from $w_{\mathcal{C}, \mathcal{O}} \lambda$. 
Then, since $\overline{w}_{\mathcal{C}, \mathcal{O}} v_\lambda$ is an extremal weight vector, it follows that 
\[V(\lambda) = V(\lambda, w_{\mathcal{C}, \mathcal{O}}) \oplus \c \overline{w}_{\mathcal{C}, \mathcal{O}} v_\lambda.\]
We define $\tau^{(\mathcal{C}, \mathcal{O})}_\lambda \in H^0(G/B, \mathcal{L}_\lambda) = V(\lambda)^\ast$ by 
\begin{equation}\label{eq:specific_section_tau}
\begin{aligned}
&\tau^{(\mathcal{C}, \mathcal{O})}_\lambda (\overline{w}_{\mathcal{C}, \mathcal{O}} v_\lambda) = 1\ \text{and}\\
&\tau^{(\mathcal{C}, \mathcal{O})}_\lambda (v) = 0\ \text{for\ all}\ v \in V(\lambda, w_{\mathcal{C}, \mathcal{O}}).
\end{aligned}
\end{equation}
For $\sigma \in H^0(G/B, \mathcal{L}_\lambda) = V(\lambda)^\ast$, define $\Upsilon_{\mathcal{C}, \mathcal{O}} (\sigma) \in \c[t_1, \ldots, t_N]$ by 
\[\Upsilon_{\mathcal{C}, \mathcal{O}} (\sigma) \coloneqq \sigma (\Omega_{\mathcal{C}, \mathcal{O}} (t_1, \ldots, t_N) v_\lambda).\]

\begin{lem}\label{l:reduction_to_action_on_highest}
For each $\lambda \in P_+$ and $\sigma \in H^0(G/B, \mathcal{L}_\lambda)$, it holds that 
\[\nu_{\mathcal{C}, \mathcal{O}}^{\rm low} (\sigma/\tau^{(\mathcal{C}, \mathcal{O})}_\lambda) = \nu_{\mathcal{C}, \mathcal{O}}^{\rm low} (\Upsilon_{\mathcal{C}, \mathcal{O}} (\sigma)).\]
\end{lem}

\begin{proof}
By the definition of $\tau^{(\mathcal{C}, \mathcal{O})}_\lambda$, the constant term of $\Upsilon_{\mathcal{C}, \mathcal{O}} (\tau^{(\mathcal{C}, \mathcal{O})}_\lambda) \in \c[t_1, \ldots, t_N]$ is given by 
\begin{align*}
\tau^{(\mathcal{C}, \mathcal{O})}_\lambda (\Omega_{\mathcal{C}, \mathcal{O}} (0, \ldots, 0) v_\lambda) = \tau^{(\mathcal{C}, \mathcal{O})}_\lambda (\overline{w}_{\mathcal{C}, \mathcal{O}} v_\lambda) = 1. 
\end{align*}
In particular, we have $\Upsilon_{\mathcal{C}, \mathcal{O}} (\tau^{(\mathcal{C}, \mathcal{O})}_\lambda) \neq 0$ and $\nu_{\mathcal{C}, \mathcal{O}}^{\rm low} (\Upsilon_{\mathcal{C}, \mathcal{O}} (\tau^{(\mathcal{C}, \mathcal{O})}_\lambda)) = (0, \ldots, 0)$.
By \eqref{eq:rational_map_computation}, it holds in $\c(G/B) = \c(t_1, \ldots, t_N)$ that 
\begin{align*}
\sigma/\tau^{(\mathcal{C}, \mathcal{O})}_\lambda = \Upsilon_{\mathcal{C}, \mathcal{O}} (\sigma)/\Upsilon_{\mathcal{C}, \mathcal{O}} (\tau^{(\mathcal{C}, \mathcal{O})}_\lambda),
\end{align*}
which implies that 
\begin{align*}
\nu_{\mathcal{C}, \mathcal{O}}^{\rm low} (\sigma/\tau^{(\mathcal{C}, \mathcal{O})}_\lambda) &= \nu_{\mathcal{C}, \mathcal{O}}^{\rm low} (\Upsilon_{\mathcal{C}, \mathcal{O}} (\sigma)) - \nu_{\mathcal{C}, \mathcal{O}}^{\rm low} (\Upsilon_{\mathcal{C}, \mathcal{O}} (\tau^{(\mathcal{C}, \mathcal{O})}_\lambda)) = \nu_{\mathcal{C}, \mathcal{O}}^{\rm low} (\Upsilon_{\mathcal{C}, \mathcal{O}} (\sigma)).
\end{align*}
This proves the lemma.
\end{proof}

\section{Main result}\label{s:main}

In this section, we realize the marked chain-order polytope $\Delta_{\mathcal{C}, \mathcal{O}}(\Pi_A, \Pi^\ast_A, \lambda)$ as a Newton--Okounkov body of $(G/B, \mathcal{L}_\lambda)$. 
In Section \ref{ss:statement}, we give the statement of our main result and its application to toric degenerations.  
The proof of our main result is given in Section \ref{ss:proof}.
Section \ref{ss:explicit basis} is devoted to constructing a specific basis of $V(\lambda)$ which is naturally parametrized by the set of lattice points in $\Delta_{\mathcal{C}, \mathcal{O}}(\Pi_A, \Pi^\ast_A, \lambda)$. 
In Section \ref{ss:highest term}, we compare our main result with the highest term valuation $\nu_{\mathcal{C}, \mathcal{O}}^{\rm high}$. 

\subsection{Statement and applications}\label{ss:statement}

The following is the main result of the present paper.

\begin{thm}\label{t:main_thm}
For each partition $\Pi_A \setminus \Pi^\ast_A = \mathcal{C} \sqcup \mathcal{O}$, $\lambda \in P_+$, and $\tau \in H^0(G/B, \mathcal{L}_\lambda) \setminus \{0\}$, the Newton--Okounkov body $\Delta(G/B, \mathcal{L}_\lambda, \nu_{\mathcal{C}, \mathcal{O}}^{\rm low}, \tau)$ coincides with the marked chain-order polytope $\Delta_{\mathcal{C}, \mathcal{O}}(\Pi_A, \Pi^\ast_A, \lambda)$ up to translations by integer vectors.
\end{thm}

We give a proof of \cref{t:main_thm} in the next subsection.
In the rest of this subsection, we see some applications of \cref{t:main_thm}.

\begin{cor}\label{c:of_main_result}
Let $\Pi_A \setminus \Pi^\ast_A = \mathcal{C} \sqcup \mathcal{O}$ be a partition, $\lambda \in P_+$, and $\tau \in H^0(G/B, \mathcal{L}_\lambda) \setminus \{0\}$.
\begin{enumerate}
\item[{\rm (1)}] The semigroup $S(G/B, \mathcal{L}_\lambda, \nu_{\mathcal{C}, \mathcal{O}}^{\rm low}, \tau)$ is finitely generated and saturated. 
\item[{\rm (2)}] The real closed cone $C(G/B, \mathcal{L}_\lambda, \nu_{\mathcal{C}, \mathcal{O}}^{\rm low}, \tau)$ is a rational convex polyhedral cone, and the following equality holds: 
\[S(G/B, \mathcal{L}_\lambda, \nu_{\mathcal{C}, \mathcal{O}}^{\rm low}, \tau) = C(G/B, \mathcal{L}_\lambda, \nu_{\mathcal{C}, \mathcal{O}}^{\rm low}, \tau) \cap (\z_{>0} \times \z^N).\]
\item[{\rm (3)}] The Newton--Okounkov body $\Delta(G/B, \mathcal{L}_\lambda, \nu_{\mathcal{C}, \mathcal{O}}^{\rm low}, \tau)$ is an integral convex polytope, and it holds that 
\[\Delta(G/B, \mathcal{L}_\lambda, \nu_{\mathcal{C}, \mathcal{O}}^{\rm low}, \tau) \cap \z^N = \{\nu_{\mathcal{C}, \mathcal{O}}^{\rm low} (\sigma/\tau) \mid \sigma \in H^0(G/B, \mathcal{L}_\lambda) \setminus \{0\}\}.\]
\end{enumerate}
\end{cor}

\begin{proof}
By \cref{t:main_thm}, there exists ${\bm a}_\lambda \in \z^N$ such that 
\[\Delta(G/B, \mathcal{L}_\lambda, \nu_{\mathcal{C}, \mathcal{O}}^{\rm low}, \tau) = \Delta_{\mathcal{C}, \mathcal{O}}(\Pi_A, \Pi^\ast_A, \lambda) - {\bm a}_\lambda,\]
which implies that $\Delta(G/B, \mathcal{L}_\lambda, \nu_{\mathcal{C}, \mathcal{O}}^{\rm low}, \tau)$ is an integral convex polytope.
By \eqref{eq:NO_scalar_multiple} and \cref{t:decomposition_property_marked_poset}, we obtain that 
\begin{align*}
\Delta(G/B, \mathcal{L}_{\lambda}^{\otimes k}, \nu_{\mathcal{C}, \mathcal{O}}^{\rm low}, \tau^k) &= k\Delta(G/B, \mathcal{L}_{\lambda}, \nu_{\mathcal{C}, \mathcal{O}}^{\rm low}, \tau)\\
&= \Delta_{\mathcal{C}, \mathcal{O}}(\Pi_A, \Pi^\ast_A, k\lambda) - k {\bm a}_\lambda
\end{align*}
for all $k \in \z_{>0}$.
Note that $\mathcal{L}_{\lambda}^{\otimes k} = \mathcal{L}_{k \lambda}$.
By \eqref{eq:marked_chain_order_number_of_lattice_points}, the number of lattice points in $\Delta_{\mathcal{C}, \mathcal{O}}(\Pi_A, \Pi^\ast_A, k\lambda)$ coincides with $\dim_\c(H^0(G/B, \mathcal{L}_{k\lambda}))$ for each $k \in \z_{>0}$.
Hence we see by \cref{p:property_valuation} (2) that 
\[\{\nu_{\mathcal{C}, \mathcal{O}}^{\rm low} (\sigma/\tau^k) \mid \sigma \in H^0(G/B, \mathcal{L}_{k\lambda}) \setminus \{0\}\} = (\Delta_{\mathcal{C}, \mathcal{O}}(\Pi_A, \Pi^\ast_A, k\lambda) \cap \z^N) - k {\bm a}_\lambda.\]
Then part (3) follows by the case $k = 1$.
In addition, we obtain that 
\[S(G/B, \mathcal{L}_\lambda, \nu_{\mathcal{C}, \mathcal{O}}^{\rm low}, \tau) = \bigcup_{k \in \z_{>0}} \{(k, {\bm a} - k {\bm a}_\lambda) \mid {\bm a} \in \Delta_{\mathcal{C}, \mathcal{O}}(\Pi_A, \Pi^\ast_A, k\lambda) \cap \z^N\}.\]
Hence we see by the definition of $\Delta_{\mathcal{C}, \mathcal{O}}(\Pi_A, \Pi^\ast_A, k\lambda)$ that $S(G/B, \mathcal{L}_\lambda, \nu_{\mathcal{C}, \mathcal{O}}^{\rm low}, \tau)$ can be written as an intersection of a rational convex polyhedral cone $\mathscr{C}$ with $\z_{>0} \times \z^N$. 
Then it follows that $\mathscr{C} = C(G/B, \mathcal{L}_\lambda, \nu_{\mathcal{C}, \mathcal{O}}^{\rm low}, \tau)$, which implies part (2). 
Finally, part (1) follows from part (2) by Gordan's lemma (see, for instance, \cite[Proposition 1.2.17]{CLS}). 
This proves the corollary.
\end{proof}

Let $P_{++} \coloneqq \sum_{i \in I} \z_{>0} \varpi_i \subseteq P_+$ be the set of regular dominant integral weights. 
For $\lambda \in P_{++}$, the line bundle $\mathcal{L}_\lambda$ on $G/B$ is very ample (see, for instance, \cite[Section I\hspace{-.1em}I.8.5]{Jan}). 
Hence it follows by \cref{t:NO_dimension_volume} that the real dimension of $\Delta(G/B, \mathcal{L}_\lambda, \nu_{\mathcal{C}, \mathcal{O}}^{\rm low}, \tau)$ coincides with $N$.
We say that $G/B$ admits a \emph{flat degeneration} to a variety $X$ if there exists a flat morphism $\pi \colon \mathfrak{X} \rightarrow {\rm Spec}(\c[t])$ of schemes such that the scheme-theoretic fiber $\pi^{-1}(t)$ (resp., $\pi^{-1}(0)$) over a closed point $t \in \c \setminus \{0\}$ (resp., the origin $0 \in \c$) is isomorphic to $G/B$ (resp., $X$). 
If $X$ is a toric variety, then a flat degeneration to $X$ is called a \emph{toric degeneration}. 
Our main result (\cref{t:main_thm}) allows us to apply Anderson's construction \cite{And} of toric degenerations to $\Delta_{\mathcal{C}, \mathcal{O}}(\Pi_A, \Pi^\ast_A, \lambda)$ as follows. 

\begin{thm}\label{t:toric_deg}
For each partition $\Pi_A \setminus \Pi^\ast_A = \mathcal{C} \sqcup \mathcal{O}$ and $\lambda \in P_{++}$, there exists a flat degeneration of $G/B$ to the irreducible normal projective toric variety corresponding to the marked chain-order polytope $\Delta_{\mathcal{C}, \mathcal{O}}(\Pi_A, \Pi^\ast_A, \lambda)$.
\end{thm}

\begin{proof}
Let $\c[S(G/B, \mathcal{L}_\lambda, \nu_{\mathcal{C}, \mathcal{O}}^{\rm low}, \tau)]$ denote the semigroup ring of $S(G/B, \mathcal{L}_\lambda, \nu_{\mathcal{C}, \mathcal{O}}^{\rm low}, \tau)$.
By \cite[Theorem 1]{And} and \cref{c:of_main_result} (1), we obtain a flat degeneration of $G/B$ to ${\rm Proj} (\c[S(G/B, \mathcal{L}_\lambda, \nu_{\mathcal{C}, \mathcal{O}}^{\rm low}, \tau)])$, where the $\z_{>0}$-grading of $S(G/B, \mathcal{L}_\lambda, \nu_{\mathcal{C}, \mathcal{O}}^{\rm low}, \tau)$ induces a $\z_{\ge 0}$-grading of $\c[S(G/B, \mathcal{L}_\lambda, \nu_{\mathcal{C}, \mathcal{O}}^{\rm low}, \tau)]$. 
By \cite[Theorem 1.3.5]{CLS} and \cref{c:of_main_result} (1), it follows that ${\rm Proj} (\c[S(G/B, \mathcal{L}_\lambda, \nu_{\mathcal{C}, \mathcal{O}}^{\rm low}, \tau)])$ is normal; hence it is isomorphic to the irreducible normal projective toric variety corresponding to the integral convex polytope $\Delta(G/B, \mathcal{L}_\lambda, \nu_{\mathcal{C}, \mathcal{O}}^{\rm low}, \tau)$. 
From these, we conclude the theorem by \cref{t:main_thm}.
\end{proof}

\subsection{Proof of \cref{t:main_thm}}\label{ss:proof}

Fix $1 \leq k \leq n$, and define 
\[{\bm d}_{\varpi_k} = (d_1^{(1)}, d_1^{(2)}, d_2^{(1)}, d_1^{(3)}, d_2^{(2)}, d_3^{(1)}, \ldots, d_1^{(n)}, d_2^{(n-1)}, \ldots, d_n^{(1)}) \in \{0, 1\}^N\] 
by
\[d_\ell^{(m)} \coloneqq 
\begin{cases} 
1 &\text{if}\ n-k+2 \leq \ell\ \text{and}\ q_\ell^{(m)} \in \mathcal{O},\\ 
0 &\text{otherwise}.
\end{cases}\]
Then we first prove that 
\[\Delta_{\mathcal{C}, \mathcal{O}}(\Pi_A, \Pi^\ast_A, \varpi_k) - {\bm d}_{\varpi_k} \subseteq \Delta(G/B, \mathcal{L}_{\varpi_k}, \nu_{\mathcal{C}, \mathcal{O}}^{\rm low}, \tau)\]
for some nonzero section $\tau \in H^0(G/B, \mathcal{L}_{\varpi_k})$.
By definition, the marked chain-order polytope $\Delta_{\mathcal{C}, \mathcal{O}}(\Pi_A, \Pi^\ast_A, \varpi_k)$ coincides with the set of 
\[(a_1^{(1)}, a_1^{(2)}, a_2^{(1)}, a_1^{(3)}, a_2^{(2)}, a_3^{(1)}, \ldots, a_1^{(n)}, a_2^{(n-1)}, \ldots, a_n^{(1)}) \in \r^N_{\geq 0}\]
satisfying the following conditions: 
\begin{itemize}
\item $a_\ell^{(m)} = 0$ if $\ell + m \leq n-k+1$,
\item $a_\ell^{(m)} = 0$ if $n-k+2 \leq \ell$ and $q_\ell^{(m)} \in \mathcal{C}$,
\item $a_\ell^{(m)} = 1$ if $n-k+2 \leq \ell$ and $q_\ell^{(m)} \in \mathcal{O}$,
\item $(a_\ell^{(m)} \mid 1 \leq \ell \leq n-k+1,\ n-k-\ell+2 \leq m \leq n+1-\ell)$ is contained in the marked chain-order polytope $\Delta_{\mathcal{C}_k, \mathcal{O}_k}(\Pi_k, \Pi^\ast_k, \mu_k)$, where the marked poset $(\Pi_k, \Pi^\ast_k, \mu_k)$ is defined by the marked Hasse diagram given in Figure \ref{type_A_marked_Hasse_Grassmann}, and the partition $\Pi_k \setminus \Pi^\ast_k = \mathcal{C}_k \sqcup \mathcal{O}_k$ is the one induced from $\Pi_A \setminus \Pi^\ast_A = \mathcal{C} \sqcup \mathcal{O}$. 
\end{itemize}
\begin{figure}[!ht]
\begin{center}
   \includegraphics[width=8.0cm,bb=60mm 140mm 150mm 230mm,clip]{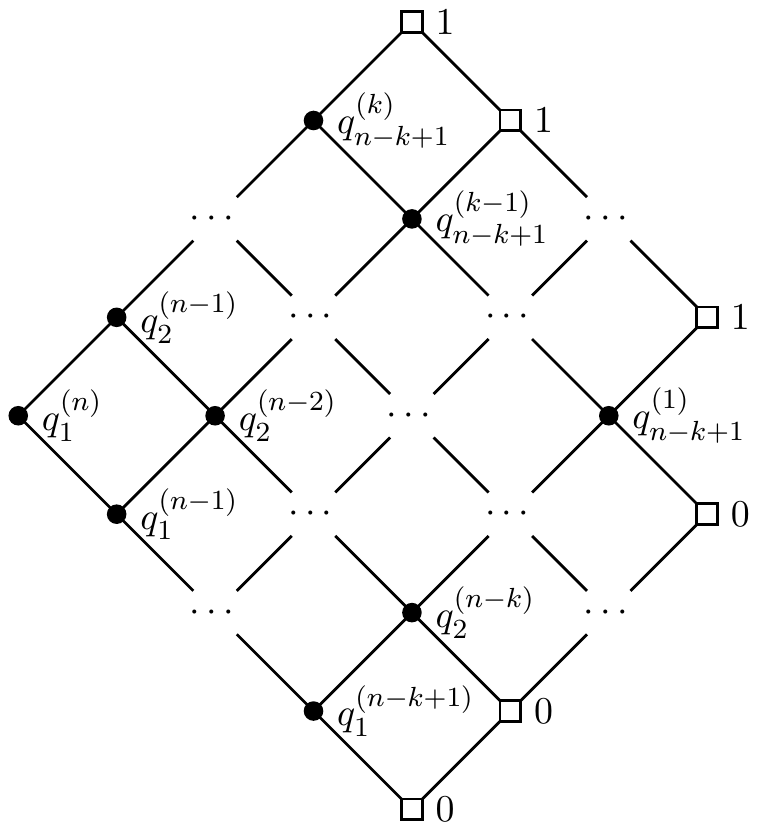}
	\caption{The marked Hasse diagram of the marked poset $(\Pi_k, \Pi^\ast_k, \mu_k)$.}
	\label{type_A_marked_Hasse_Grassmann}
\end{center}
\end{figure}
For $b \in \mathcal{B}(\varpi_k)$, it follows that 
\begin{align*}
\Upsilon_{\mathcal{C}, \mathcal{O}} (G^{\rm up}_{\varpi_k} (b)) &= G^{\rm up}_{\varpi_k} (b) (\Omega_{\mathcal{C}, \mathcal{O}} (t_1, \ldots, t_N) v_{\varpi_k})\\
&= \sum_{{\bm a} = (a_1, \ldots, a_N) \in \z_{\geq 0}^N} \frac{1}{a_1! \cdots a_N!} t_1^{a_1} \cdots t_N^{a_N} G^{\rm up}_{\varpi_k} (b) (\overline{u_1} f_{i_1}^{a_1} \cdots \overline{u_N} f_{i_N}^{a_N} v_{\varpi_k}).
\end{align*}
In addition, we have $G^{\rm up}_{\varpi_k} (b) (\overline{u_1} f_{i_1}^{a_1} \cdots \overline{u_N} f_{i_N}^{a_N} v_{\varpi_k}) \neq 0$ if and only if $u_1 \tilde{f}_{i_1}^{a_1} \cdots u_N \tilde{f}_{i_N}^{a_N} b_{\varpi_k} = b$, where $e b^\prime \coloneqq b^\prime$ for the identity element $e \in W$ and $b^\prime \in \mathcal{B}(\varpi_k)$.
In particular, the following holds. 

\begin{lem}\label{l:values_are_distinct}
The values $\nu_{\mathcal{C}, \mathcal{O}}^{\rm low} (\Upsilon_{\mathcal{C}, \mathcal{O}} (G^{\rm up}_{\varpi_k} (b)))$, $b \in \mathcal{B}(\varpi_k)$, are all distinct. 
\end{lem}

Fix $b \in \mathcal{B}(\varpi_k)$, and write $\nu_{\mathcal{C}, \mathcal{O}}^{\rm low} (\Upsilon_{\mathcal{C}, \mathcal{O}} (G^{\rm up}_{\varpi_k} (b))) = (a_1, \ldots, a_N)$.
Let us prove that $(a_1, \ldots, a_N) \in \Delta_{\mathcal{C}, \mathcal{O}}(\Pi_A, \Pi^\ast_A, \varpi_k) -{\bm d}_{\varpi_k}$. 
Since we have $G^{\rm up}_{\varpi_k} (b) (\overline{u_1} f_{i_1}^{a_1} \cdots \overline{u_N} f_{i_N}^{a_N} v_{\varpi_k}) \neq 0$, it follows that 
\[u_1 \tilde{f}_{i_1}^{a_1} \cdots u_N \tilde{f}_{i_N}^{a_N} b_{\varpi_k} = b,\] 
and hence that 
\[b_{\varpi_k} = \tilde{e}_{i_N}^{a_N} u_N \cdots \tilde{e}_{i_1}^{a_1} u_1 b.\]
For $0 \leq \ell \leq N$, define $b^{(\ell)} \in \mathcal{B}(\varpi_k)$ by 
\[b^{(\ell)} \coloneqq \tilde{e}_{i_\ell}^{a_\ell} u_\ell \cdots \tilde{e}_{i_1}^{a_1} u_1 b.\] 
In particular, we have $b^{(0)} = b$ and $b^{(N)} = b_{\varpi_k}$.
Let us write 
\[b^{(\ell)} = (j_1^{(\ell)}, j_2^{(\ell)}, \ldots, j_k^{(\ell)})\]
for $0 \leq \ell \leq N$ as in Section \ref{s:crystal_basis}. 
Recalling the arrangement \eqref{eq:arrangement_of_poset} of the elements of $\Pi_A \setminus \Pi^\ast_A$, define $1 \leq z_\ell \leq n +1 -i_\ell$ for $1 \leq \ell \leq N$ by the condition that $q_\ell = q_{z_\ell}^{(i_\ell)}$.  
Then we write $a_{z_\ell}^{(i_\ell)} \coloneqq a_\ell$, $b_{z_\ell}^{(i_\ell)} \coloneqq b^{(\ell)}$, and $u_{z_\ell}^{(i_\ell)} \coloneqq u_\ell$ for $1 \leq \ell \leq N$.
Set $N_m \coloneqq \frac{m(m+1)}{2}$ for $0 \leq m \leq n$. 
In particular, we have $N_n = N$.

\begin{lem}\label{l:partial_form_of_tableau_induction}
For $n-k+1 \leq m \leq n$ and $1 \leq q \leq m - (n-k)$, the equality $j_q^{(N_m)} = q$ holds.
\end{lem}

\begin{proof}
Since we have $1 \leq j_1^{(N_m)} < \cdots < j_k^{(N_m)}$, it suffices to show that $j_{m - (n-k)}^{(N_m)} \leq m - (n-k)$ for all $n-k+1 \leq m \leq n$. 
Assume for a contradiction that 
\begin{equation}\label{eq:partial_form_tableau_1}
\begin{aligned}
j_{m - (n-k)}^{(N_m)} > m - (n-k)
\end{aligned}
\end{equation} 
for some $n-k+1 \leq m \leq n$. 
If $m = n$, then it follows by \eqref{eq:partial_form_tableau_1} that 
\begin{equation}\label{eq:partial_form_tableau_contradiction}
\begin{aligned}
j_k^{(N)} = j_{n - (n-k)}^{(N_n)} > n - (n-k) = k,
\end{aligned}
\end{equation} 
which contradicts to the equality $b^{(N)} = b_{\varpi_k}$. 
Hence we may assume that $m < n$. 
Then, since we have 
\[b^{(N_m + n - k +1)} = \tilde{e}_{i_{N_m + n - k +1}}^{a_{N_m + n - k +1}} u_{N_m + n - k +1} \cdots \tilde{e}_{i_{N_m + 1}}^{a_{N_m + 1}} u_{N_m + 1} b^{(N_m)}\] 
and 
\[(i_{N_m + n - k +1}, i_{N_m + n - k}, \ldots, i_{N_m + 1}) = (m - (n-k) +1, m - (n-k) +2, \ldots, m + 1),\] 
it follows by \eqref{eq:partial_form_tableau_1} that $j_{m - (n-k)}^{(N_m + n - k +1)} > m - (n-k)$, and hence that 
\begin{equation}\label{eq:partial_form_tableau_2}
\begin{aligned}
j_{m - (n-k) + 1}^{(N_m + n - k +1)} > m - (n-k) + 1.
\end{aligned}
\end{equation}  
Since we have 
\[b^{(N_{m+1})} = \tilde{e}_{i_{N_{m+1}}}^{a_{N_{m+1}}} u_{N_{m+1}} \cdots \tilde{e}_{i_{N_m + n - k +2}}^{a_{N_m + n - k +2}} u_{N_m + n - k +2} b^{(N_m + n - k +1)}\] 
and 
\[(i_{N_{m+1}}, i_{N_{m+1} - 1}, \ldots, i_{N_m + n - k +2}) = (1, 2, \ldots, m - (n-k)),\] 
it holds by \eqref{eq:partial_form_tableau_2} that $j_{m - (n-k) + 1}^{(N_{m+1})} > m - (n-k) + 1$. 
Repeating this argument, we obtain \eqref{eq:partial_form_tableau_contradiction} which gives a contradiction. 
This proves the lemma. 
\end{proof}

Let us define $(\hat{a}_1, \hat{a}_2, \ldots, \hat{a}_N) \in \{0, 1\}^N$ as follows. 
We first set $\hat{a}_\ell \coloneqq 0$ for all $1 \leq \ell \leq N_{n-k}$. 
Then, assuming that $(\hat{a}_1, \hat{a}_2, \ldots, \hat{a}_{N_{m-1}})$ is defined for some $n-k+1 \leq m \leq n$, let us determine $(\hat{a}_{N_{m-1}+1}, \hat{a}_{N_{m-1}+2}, \ldots, \hat{a}_{N_m})$. 
To do that, we define $\hat{b}^{(\ell)} \in \mathcal{B}(\varpi_k)$ for $0 \leq \ell \leq N$ by
\[\hat{b}^{(\ell)} \coloneqq \tilde{e}_{i_\ell}^{\hat{a}_{\ell}} u_\ell \cdots \tilde{e}_{i_1}^{\hat{a}_1} u_1 b\]
under the assumption that $\hat{a}_1, \ldots, \hat{a}_\ell$ are defined, and write 
\begin{align*}
&\hat{b}^{(\ell)} = (\hat{j}_1^{(\ell)}, \hat{j}_2^{(\ell)}, \ldots, \hat{j}_k^{(\ell)}).
\end{align*}
Then let us set
\begin{itemize}
\item $\hat{a}_{N_{m-1}+\ell} = 0$ for $1 \leq \ell \leq m-\gamma_m +1$ such that $q_{N_{m-1}+\ell} \in \mathcal{O}$,
\item $\hat{a}_{N_{m-1}+\ell} = 1$ for $m-\gamma_m +2 \leq \ell \leq n-k+1$ such that $q_{N_{m-1}+\ell} \in \mathcal{O}$,
\item $\hat{a}_{N_{m-1}+\ell} = 0$ for $n-k+2 \leq \ell \leq m$ such that $q_{N_{m-1}+\ell} \in \mathcal{O}$,
\item $\hat{a}_{N_{m-1}+\ell} = 0$ for $1 \leq \ell \leq m$ such that $\ell \neq m-\gamma_m +1$ and $q_{N_{m-1}+\ell} \in \mathcal{C}$,
\item $\hat{a}_{N_{m-1} + m-\gamma_m +1} = 0$ if $\hat{j}_{m - (n-k) +1}^{(N_{m-1} + m-\gamma_m)} = \gamma_m + 1$ and $q_{N_{m-1} +m -\gamma_m +1} \in \mathcal{C}$,
\item $\hat{a}_{N_{m-1} + m-\gamma_m +1} = 1$ if $\hat{j}_{m - (n-k) +1}^{(N_{m-1} + m-\gamma_m)} \neq \gamma_m + 1$ and $q_{N_{m-1} +m -\gamma_m +1} \in \mathcal{C}$,
\end{itemize}
where $\gamma_m \coloneqq \hat{j}_{m - (n-k)}^{(N_{m-1})}$.

\begin{lem}\label{l:explicit_computation_of_values_for_minors}
The equality $a_\ell = \hat{a}_\ell$ holds for all $1 \leq \ell \leq N$.
\end{lem}

\begin{proof}
Since we have 
\[\hat{b}^{(N_{n-k+1})} = \tilde{e}_{i_{N_{n-k+1}}}^{\hat{a}_{N_{n-k+1}}} u_{N_{n-k+1}} \cdots \tilde{e}_{i_{N_{n-k}+1}}^{\hat{a}_{N_{n-k}+1}} u_{N_{n-k}+1} \hat{b}^{(N_{n-k})},\] 
it follows by the definition of $(\hat{a}_{N_{n-k}+1}, \hat{a}_{N_{n-k}+2}, \ldots, \hat{a}_{N_{n-k+1}})$ that $\hat{j}_1^{(N_{n-k+1})} = 1$. 
Then, since it holds that 
\[\hat{b}^{(N_{n-k+2})} = \tilde{e}_{i_{N_{n-k+2}}}^{\hat{a}_{N_{n-k+2}}} u_{N_{n-k+2}} \cdots \tilde{e}_{i_{N_{n-k+1}+1}}^{\hat{a}_{N_{n-k+1}+1}} u_{N_{n-k+1}+1} \hat{b}^{(N_{n-k+1})},\] 
we see by the definition of $(\hat{a}_{N_{n-k+1}+1}, \hat{a}_{N_{n-k+1}+2}, \ldots, \hat{a}_{N_{n-k+2}})$ that $\hat{j}_1^{(N_{n-k+2})} = 1$ and $\hat{j}_2^{(N_{n-k+2})} = 2$.
Continuing in this way, we deduce that 
\[\hat{b}^{(N)} = \tilde{e}_{i_N}^{\hat{a}_N} u_N \cdots \tilde{e}_{i_1}^{\hat{a}_1} u_1 b = (1, 2, \ldots, k) = b_{\varpi_k}.\]
Hence the definition of $\nu_{\mathcal{C}, \mathcal{O}}^{\rm low} (\Upsilon_{\mathcal{C}, \mathcal{O}} (G^{\rm up}_{\varpi_k} (b))) = (a_1, \ldots, a_N)$ implies that 
\[(a_1, a_2, \ldots, a_N) \leq (\hat{a}_1, \hat{a}_2, \ldots, \hat{a}_N)\] 
with respect to the lexicographic order. 
In particular, we have $a_\ell = 0$ for all $1 \leq \ell \leq N_{n-k}$, which implies that 
\[b^{(N_{n-k})} = u_{N_{n-k}} \cdots u_1 b = \hat{b}^{(N_{n-k})}\]
and that 
\[(a_{N_{n-k}+1}, a_{N_{n-k}+2}, \ldots, a_{N_{n-k+1}}) \leq (\hat{a}_{N_{n-k}+1}, \hat{a}_{N_{n-k}+2}, \ldots, \hat{a}_{N_{n-k+1}}).\]
If $(a_{N_{n-k}+1}, a_{N_{n-k}+2}, \ldots, a_{N_{n-k+1}}) < (\hat{a}_{N_{n-k}+1}, \hat{a}_{N_{n-k}+2}, \ldots, \hat{a}_{N_{n-k+1}})$, then it is easy to see that $j_1^{(N_{n-k+1})} \neq 1$ by the definition of $(\hat{a}_{N_{n-k}+1}, \hat{a}_{N_{n-k}+2}, \ldots, \hat{a}_{N_{n-k+1}})$ since we have 
\begin{align*}
(j_1^{(N_{n-k+1})}, j_2^{(N_{n-k+1})}, \ldots, j_k^{(N_{n-k+1})}) &= b^{(N_{n-k+1})}\\ 
&= \tilde{e}_{i_{N_{n-k+1}}}^{a_{N_{n-k+1}}} u_{N_{n-k+1}} \cdots \tilde{e}_{i_{N_{n-k}+1}}^{a_{N_{n-k}+1}} u_{N_{n-k}+1} b^{(N_{n-k})}\\
&= \tilde{e}_{i_{N_{n-k+1}}}^{a_{N_{n-k+1}}} u_{N_{n-k+1}} \cdots \tilde{e}_{i_{N_{n-k}+1}}^{a_{N_{n-k}+1}} u_{N_{n-k}+1} \hat{b}^{(N_{n-k})}.
\end{align*}
This gives a contradiction since we have $j_1^{(N_{n-k+1})} = 1$ by \cref{l:partial_form_of_tableau_induction}. 
Hence it follows that 
\[(a_{N_{n-k}+1}, a_{N_{n-k}+2}, \ldots, a_{N_{n-k+1}}) = (\hat{a}_{N_{n-k}+1}, \hat{a}_{N_{n-k}+2}, \ldots, \hat{a}_{N_{n-k+1}}).\] 
Repeating this argument, we conclude the assertion of the lemma.
\end{proof}

\begin{lem}\label{l:forms_of_1_in_chain_coordinates}
It holds that $\gamma_{m+1} > \gamma_m$ for all $n-k+1 \leq m \leq n-1$. 
In addition, if $\gamma_{m+1} = \gamma_m +1$ for some $n-k+1 \leq m \leq n-1$, then $\hat{a}_{N_{m-1} +m -\gamma_m +1} = 0$.
\end{lem}

\begin{proof}
Since we have
\[(i_{N_{m-1} + 1}, i_{N_{m-1} + 2}, \ldots, i_{N_m}) = (m, m-1, \ldots, 1),\]
it follows that
\[\gamma_m = \hat{j}_{m - (n-k)}^{(N_{m-1})} = \hat{j}_{m - (n-k)}^{(N_{m-1} + m -\gamma_m)} < \hat{j}_{m - (n-k) +1}^{(N_{m-1} + m -\gamma_m)} = \hat{j}_{m - (n-k) +1}^{(N_m)} = \gamma_{m+1}.\]
In addition, if $\gamma_{m+1} = \gamma_m +1$, then we have $\hat{j}_{m - (n-k) +1}^{(N_{m-1} + m -\gamma_m)} = \gamma_m + 1$, which implies that $\hat{a}_{N_{m-1} +m -\gamma_m +1} = 0$ by definition.
This proves the lemma.
\end{proof}

By \cref{l:forms_of_1_in_chain_coordinates} and the definition of $(\hat{a}_1, \hat{a}_2, \ldots, \hat{a}_N)$, we deduce that
\[(\hat{a}_1, \hat{a}_2, \ldots, \hat{a}_N) \in \Delta_{\mathcal{C}, \mathcal{O}}(\Pi_A, \Pi^\ast_A, \varpi_k)-{\bm d}_{\varpi_k},\]
which implies by \cref{l:explicit_computation_of_values_for_minors} that 
\[\nu_{\mathcal{C}, \mathcal{O}}^{\rm low} (\Upsilon_{\mathcal{C}, \mathcal{O}} (G^{\rm up}_{\varpi_k} (b))) = (a_1, a_2, \ldots, a_N) \in \Delta_{\mathcal{C}, \mathcal{O}}(\Pi_A, \Pi^\ast_A, \varpi_k)-{\bm d}_{\varpi_k}.\]
Since the number of lattice points in $\Delta_{\mathcal{C}, \mathcal{O}}(\Pi_A, \Pi^\ast_A, \varpi_k)$ coincides with the cardinality of $\mathcal{B}(\varpi_k)$ by \eqref{eq:marked_chain_order_number_of_lattice_points}, it follows by \cref{l:values_are_distinct} that 
\[(\Delta_{\mathcal{C}, \mathcal{O}}(\Pi_A, \Pi^\ast_A, \varpi_k)-{\bm d}_{\varpi_k}) \cap \z^N = \{\nu_{\mathcal{C}, \mathcal{O}}^{\rm low} (\Upsilon_{\mathcal{C}, \mathcal{O}} (G^{\rm up}_{\varpi_k} (b))) \mid b \in \mathcal{B}(\varpi_k)\},\]
and hence that 
\begin{align*}
\Delta_{\mathcal{C}, \mathcal{O}}(\Pi_A, \Pi^\ast_A, \varpi_k)-{\bm d}_{\varpi_k} &= {\rm Conv}((\Delta_{\mathcal{C}, \mathcal{O}}(\Pi_A, \Pi^\ast_A, \varpi_k)-{\bm d}_{\varpi_k}) \cap \z^N)\\
&= {\rm Conv}(\{\nu_{\mathcal{C}, \mathcal{O}}^{\rm low} (\Upsilon_{\mathcal{C}, \mathcal{O}} (G^{\rm up}_{\varpi_k} (b))) \mid b \in \mathcal{B}(\varpi_k)\})\\
&\subseteq \Delta(G/B, \mathcal{L}_{\varpi_k}, \nu_{\mathcal{C}, \mathcal{O}}^{\rm low}, \tau^{(\mathcal{C}, \mathcal{O})}_{\varpi_k})\quad (\text{by Lemma}\ \ref{l:reduction_to_action_on_highest}).
\end{align*}
For $\lambda = \lambda_1 \varpi_1 + \cdots + \lambda_n \varpi_n \in P_+$, the definition of $\tau^{(\mathcal{C}, \mathcal{O})}_{\lambda}$ implies that 
\begin{align*}
&\tau^{(\mathcal{C}, \mathcal{O})}_{\lambda} = (\tau^{(\mathcal{C}, \mathcal{O})}_{\varpi_1})^{\lambda_1} \cdots (\tau^{(\mathcal{C}, \mathcal{O})}_{\varpi_n})^{\lambda_n}.
\end{align*}
Writing 
\[{\bm d}_{\lambda} \coloneqq \lambda_1 {\bm d}_{\varpi_1} + \cdots + \lambda_n {\bm d}_{\varpi_n},\] 
it follows by \cref{t:decomposition_property_marked_poset} that 
\begin{align*}
\Delta_{\mathcal{C}, \mathcal{O}}(\Pi_A, \Pi^\ast_A, \lambda)-{\bm d}_{\lambda} &= \lambda_1 (\Delta_{\mathcal{C}, \mathcal{O}}(\Pi_A, \Pi^\ast_A, \varpi_1)-{\bm d}_{\varpi_1}) + \cdots + \lambda_n (\Delta_{\mathcal{C}, \mathcal{O}}(\Pi_A, \Pi^\ast_A, \varpi_n)-{\bm d}_{\varpi_n})\\
&\subseteq \lambda_1 \Delta(G/B, \mathcal{L}_{\varpi_1}, \nu_{\mathcal{C}, \mathcal{O}}^{\rm low}, \tau^{(\mathcal{C}, \mathcal{O})}_{\varpi_1}) + \cdots + \lambda_n \Delta(G/B, \mathcal{L}_{\varpi_n}, \nu_{\mathcal{C}, \mathcal{O}}^{\rm low}, \tau^{(\mathcal{C}, \mathcal{O})}_{\varpi_n})\\
&\subseteq \Delta(G/B, \mathcal{L}_{\lambda}, \nu_{\mathcal{C}, \mathcal{O}}^{\rm low}, \tau^{(\mathcal{C}, \mathcal{O})}_{\lambda})\quad (\text{by}\ \eqref{eq:super_additivity_NO}). 
\end{align*}
Since the volumes of $\Delta_{\mathcal{C}, \mathcal{O}}(\Pi_A, \Pi^\ast_A, \lambda)$ and $\Delta(G/B, \mathcal{L}_{\lambda}, \nu_{\mathcal{C}, \mathcal{O}}^{\rm low}, \tau^{(\mathcal{C}, \mathcal{O})}_{\lambda})$ coincide by Theorems \ref{t:marked_chain-order_Ehrhart} and \ref{t:NO_dimension_volume}, we deduce that 
\begin{equation}\label{eq:main_result}
\begin{aligned}
\Delta_{\mathcal{C}, \mathcal{O}}(\Pi_A, \Pi^\ast_A, \lambda)-{\bm d}_{\lambda} = \Delta(G/B, \mathcal{L}_{\lambda}, \nu_{\mathcal{C}, \mathcal{O}}^{\rm low}, \tau^{(\mathcal{C}, \mathcal{O})}_{\lambda}).
\end{aligned}
\end{equation}
This completes the proof of \cref{t:main_thm}.

\subsection{Bases parametrized by lattice points in marked chain-order polytopes}\label{ss:explicit basis}

For each partition $\Pi_A \setminus \Pi^\ast_A = \mathcal{C} \sqcup \mathcal{O}$ and $\lambda \in P_+$, we write 
\[\widehat{\Delta}_{\mathcal{C}, \mathcal{O}}(\lambda) \coloneqq \Delta_{\mathcal{C}, \mathcal{O}}(\Pi_A, \Pi^\ast_A, \lambda)-{\bm d}_{\lambda}.\]
Then it follows by \eqref{eq:main_result} that $\widehat{\Delta}_{\mathcal{C}, \mathcal{O}}(\lambda) = \Delta(G/B, \mathcal{L}_{\lambda}, \nu_{\mathcal{C}, \mathcal{O}}^{\rm low}, \tau^{(\mathcal{C}, \mathcal{O})}_{\lambda})$.
Note that $\widehat{\Delta}_{\mathcal{C}, \mathcal{O}}(\lambda) \subseteq \r^N_{\geq 0}$ by the definition of ${\bm d}_{\lambda}$. 
For each ${\bm a} = (a_1, \ldots, a_N) \in \z^N_{\geq 0}$, define $v_\lambda^{(\mathcal{C}, \mathcal{O})}({\bm a}) \in V(\lambda)$ by 
\[v_\lambda^{(\mathcal{C}, \mathcal{O})}({\bm a}) \coloneqq \overline{u_1} f_{i_1}^{(a_1)} \cdots \overline{u_N} f_{i_N}^{(a_N)} v_\lambda,\]
where $f_i^{(a)}$ denotes the divided power $\frac{f_i^a}{a!}$ for $i \in I$ and $a \in \z_{\geq 0}$. 
Now we obtain a specific $\c$-basis of $V(\lambda)$ parametrized by the set of lattice points in $\widehat{\Delta}_{\mathcal{C}, \mathcal{O}}(\lambda)$ as follows. 

\begin{thm}\label{t:explicit_basis}
For each partition $\Pi_A \setminus \Pi^\ast_A = \mathcal{C} \sqcup \mathcal{O}$ and $\lambda \in P_+$, the set $\{v_\lambda^{(\mathcal{C}, \mathcal{O})}({\bm a}) \mid {\bm a} \in \widehat{\Delta}_{\mathcal{C}, \mathcal{O}}(\lambda) \cap \z^N\}$ forms a $\c$-basis of $V(\lambda)$. 
\end{thm}

\begin{proof}
It follows by \cref{p:property_valuation} and \cref{c:of_main_result} (3) that there exists a $\c$-basis $\{\sigma_{\bm a} \mid {\bm a} \in \widehat{\Delta}_{\mathcal{C}, \mathcal{O}}(\lambda) \cap \z^N\}$ of $H^0(G/B, \mathcal{L}_{\lambda})$ such that $\nu_{\mathcal{C}, \mathcal{O}}^{\rm low} (\sigma_{\bm a}/\tau^{(\mathcal{C}, \mathcal{O})}_{\lambda}) = {\bm a}$ for all ${\bm a} \in \widehat{\Delta}_{\mathcal{C}, \mathcal{O}}(\lambda) \cap \z^N$. 
Assume for a contradiction that the set $\{v_\lambda^{(\mathcal{C}, \mathcal{O})}({\bm a}) \mid {\bm a} \in \widehat{\Delta}_{\mathcal{C}, \mathcal{O}}(\lambda) \cap \z^N\}$ is linearly dependent.
Then there exist $\ell \in \z_{>0}$, ${\bm a}_1, \ldots, {\bm a}_\ell \in \widehat{\Delta}_{\mathcal{C}, \mathcal{O}}(\lambda) \cap \z^N$, and $c_1, \ldots, c_\ell \in \c^\times$ such that ${\bm a}_1, \ldots, {\bm a}_\ell$ are all distinct and such that
\begin{equation}\label{eq:linearly_dependent}
\begin{aligned}
c_1 v_\lambda^{(\mathcal{C}, \mathcal{O})}({\bm a}_1) + \cdots + c_\ell v_\lambda^{(\mathcal{C}, \mathcal{O})}({\bm a}_\ell) = 0.
\end{aligned}
\end{equation}
Without loss of generality, we may assume that ${\bm a}_1 > {\bm a}_2 > \cdots > {\bm a}_\ell$ with respect to the lexicographic order. 
Note that 
\begin{equation}\label{eq:coefficients_from_bases}
\begin{aligned}
\Upsilon_{\mathcal{C}, \mathcal{O}} (\sigma_{{\bm a}_1})&= \sigma_{{\bm a}_1} (\Omega_{\mathcal{C}, \mathcal{O}} (t_1, \ldots, t_N) v_\lambda)\\
&= \sum_{{\bm a} = (a_1, \ldots, a_N) \in \z_{\geq 0}^N} \sigma_{{\bm a}_1} (v_\lambda^{(\mathcal{C}, \mathcal{O})}({\bm a})) t_1^{a_1} \cdots t_N^{a_N}.
\end{aligned}
\end{equation}
Since we have $\nu_{\mathcal{C}, \mathcal{O}}^{\rm low} (\Upsilon_{\mathcal{C}, \mathcal{O}} (\sigma_{{\bm a}_1})) = \nu_{\mathcal{C}, \mathcal{O}}^{\rm low} (\sigma_{{\bm a}_1}/\tau^{(\mathcal{C}, \mathcal{O})}_{\lambda}) = {\bm a}_1$ by \cref{l:reduction_to_action_on_highest}, it follows by \eqref{eq:coefficients_from_bases} that $\sigma_{{\bm a}_1} (v_\lambda^{(\mathcal{C}, \mathcal{O})}({\bm a}_1)) \neq 0$ and that $\sigma_{{\bm a}_1} (v_\lambda^{(\mathcal{C}, \mathcal{O})}({\bm a})) = 0$ for all ${\bm a} < {\bm a}_1$. 
In particular, we have $\sigma_{{\bm a}_1} (v_\lambda^{(\mathcal{C}, \mathcal{O})}({\bm a}_2)) = \cdots = \sigma_{{\bm a}_1} (v_\lambda^{(\mathcal{C}, \mathcal{O})}({\bm a}_\ell)) = 0$, which implies that 
\begin{align*}
\sigma_{{\bm a}_1} (c_1 v_\lambda^{(\mathcal{C}, \mathcal{O})}({\bm a}_1) + \cdots + c_\ell v_\lambda^{(\mathcal{C}, \mathcal{O})}({\bm a}_\ell)) = c_1 \sigma_{{\bm a}_1} (v_\lambda^{(\mathcal{C}, \mathcal{O})}({\bm a}_1)). 
\end{align*}
Since $\sigma_{{\bm a}_1} (v_\lambda^{(\mathcal{C}, \mathcal{O})}({\bm a}_1)) \neq 0$, we conclude by \eqref{eq:linearly_dependent} that $c_1 = 0$, which gives a contradiction. 
From this, we know that the set $\{v_\lambda^{(\mathcal{C}, \mathcal{O})}({\bm a}) \mid {\bm a} \in \widehat{\Delta}_{\mathcal{C}, \mathcal{O}}(\lambda) \cap \z^N\}$ is linearly independent. 
Then the assertion of the theorem follows by \eqref{eq:marked_chain_order_number_of_lattice_points}.
\end{proof}

Let $\{\alpha_i \mid i \in I\} \subseteq P$ be the set of simple roots,
\[\Phi_+ \coloneqq \{\alpha_i + \alpha_{i+1} + \cdots + \alpha_j \mid 1 \leq i \leq j \leq n\} \subseteq P\] 
the set of positive roots, and $f_\beta \in \mathfrak{g}$ a negative root vector corresponding to $\beta \in \Phi_+$. 

\begin{rem}
Our basis $\{v_\lambda^{(\mathcal{C}, \mathcal{O})}({\bm a}) \mid {\bm a} \in \widehat{\Delta}_{\mathcal{C}, \mathcal{O}}(\lambda) \cap \z^N\}$ of $V(\lambda)$ is an analog of an \emph{essential basis} introduced by Feigin--Fourier--Littelmann \cite{FeFL1, FeFL3} and by Fang--Fourier--Littelmann \cite[Section 9.1]{FaFL}.
Following \cite[Definition 2]{FaFL}, a sequence $S = (\beta_1, \ldots, \beta_N) \in \Phi_+^N$ is called a \emph{birational sequence} if the following morphism is birational:
\[\c^N \rightarrow G/B,\quad (t_1, \ldots, t_N) \mapsto \exp(t_1 f_{\beta_1}) \cdots \exp(t_N f_{\beta_N}) \bmod B.\]  
An essential basis of $V(\lambda)$ is a $\c$-basis consisting of \emph{essential vectors} defined from a birational sequence $S$ and from a specific total order on $\z_{\geq 0}^N$ (see \cite[Definition 7 and Remark 5]{FaFL}). 
It is parametrized by the set of lattice points in the corresponding Newton--Okounkov body of $(G/B, \mathcal{L}_\lambda)$ under some conditions (see \cite[Remark 9 and Proposition 8]{FaFL}). 
Note that a birational sequence $S = (\beta_1, \ldots, \beta_N)$ corresponds to a birational morphism $\c^N \rightarrow G/B$ given as a product of $1$-parameter subgroups corresponding to $f_{\beta_1}, \ldots, f_{\beta_N}$, while our birational morphism $\widehat{\Omega}_{\mathcal{C}, \mathcal{O}} \colon \c^N \rightarrow G/B$ is defined to be a product of such $1$-parameter subgroups with translations by $\overline{u_1}, \ldots, \overline{u_N}$.
The definition of essential bases in \cite[Section 9.1]{FaFL} can be generalized to a birational morphism of the form
\[\c^N \rightarrow G/B,\quad (t_1, \ldots, t_N) \mapsto \overline{w_1} \exp(t_1 f_{\beta_1}) \cdots \overline{w_N} \exp(t_N f_{\beta_N}),\] 
for some $w_1, \ldots, w_N \in W$ and $\beta_1, \ldots, \beta_N \in \Phi_+$. 
Then our basis in \cref{t:explicit_basis} can be understood as such \emph{generalized} essential basis.
\end{rem}

\begin{ex}
If $\mathcal{C} = \emptyset$ and $\mathcal{O} = \Pi_A \setminus \Pi^\ast_A$, then we have
\[v_\lambda^{(\emptyset, \Pi_A \setminus \Pi^\ast_A)}({\bm a}) = f_{i_1}^{(a_1)} f_{i_2}^{(a_2)} \cdots f_{i_N}^{(a_N)} v_\lambda\]
for ${\bm a} = (a_1, \ldots, a_N) \in \z^N_{\geq 0}$, and the $\c$-basis $\{v_\lambda^{(\emptyset, \Pi_A \setminus \Pi^\ast_A)}({\bm a}) \mid {\bm a} \in \widehat{\Delta}_{\emptyset, \Pi_A \setminus \Pi^\ast_A}(\lambda) \cap \z^N\}$ of $V(\lambda)$ coincides with the essential basis in \cite[Section 9.1]{FaFL} associated with the birational sequence $S \coloneqq (\alpha_{i_1}, \alpha_{i_2}, \ldots, \alpha_{i_N})$ and with the \emph{homogeneous lexicographic order} on $\z_{\geq 0}^N$ (see \cite[Example 9]{FaFL}).
\end{ex}

\begin{ex}
If $\mathcal{C} = \Pi_A \setminus \Pi^\ast_A$ and $\mathcal{O} = \emptyset$, then we have
\begin{align*}
v_\lambda^{(\Pi_A \setminus \Pi^\ast_A, \emptyset)}({\bm a}) &= \overline{s}_{i_1} f_{i_1}^{(a_1)} \overline{s}_{i_2} f_{i_2}^{(a_2)} \cdots \overline{s}_{i_N} f_{i_N}^{(a_N)} v_\lambda\\
&\in \c^{\times} \overline{w_0} f_{\beta_1}^{a_1} f_{\beta_2}^{a_2} \cdots f_{\beta_N}^{a_N} v_\lambda,
\end{align*}
where we write $\beta_\ell \coloneqq s_{i_N} s_{i_{N-1}} \cdots s_{i_{\ell+1}} (\alpha_{i_\ell}) \in \Phi_+$ for $1 \leq \ell \leq N$.
Since we have 
\[(\beta_N, \ldots, \beta_1) = (\alpha_1, \alpha_1 +\alpha_2, \ldots, \alpha_1 + \cdots + \alpha_n, \alpha_2, \alpha_2 +\alpha_3, \ldots, \alpha_{n-1}, \alpha_{n-1} +\alpha_n, \alpha_n),\]
the $\c$-basis $\{v_\lambda^{(\Pi_A \setminus \Pi^\ast_A, \emptyset)}({\bm a}) \mid {\bm a} \in \widehat{\Delta}_{\Pi_A \setminus \Pi^\ast_A, \emptyset}(\lambda) \cap \z^N\}$ of $V(\lambda)$ coincides with a $\c$-basis given in \cite[Theorem 3]{FeFL1} up to the action by $\overline{w_0}$ and scalar multiples.
\end{ex}

\subsection{Comparison with highest term valuations}\label{ss:highest term}

In this subsection, we study the highest term valuation $\nu_{\mathcal{C}, \mathcal{O}}^{\rm high}$ defined in Section \ref{ss:NO_body_flag} from a partition $\Pi_A \setminus \Pi^\ast_A = \mathcal{C} \sqcup \mathcal{O}$.
When $\mathcal{C} = \emptyset$ and $\mathcal{O} = \Pi_A \setminus \Pi^\ast_A$, then the highest term valuation $\nu_{\emptyset, \Pi_A \setminus \Pi^\ast_A}^{\rm high}$ is studied by Kaveh \cite{Kav} who proved in \cite[Corollary 4.2]{Kav} that the Newton--Okounkov body $\Delta(G/B, \mathcal{L}_\lambda, \nu_{\emptyset, \Pi_A \setminus \Pi^\ast_A}^{\rm high}, \tau_\lambda)$ is unimodularly equivalent to the string polytope associated with ${\bm i}_A$ and $\lambda$. 
This string polytope is unimodularly equivalent to the Gelfand--Tsetlin polytope $GT(\lambda)$ by \cite[Corollary 5]{Lit}.
When $\mathcal{C} = \Pi_A \setminus \Pi^\ast_A$ and $\mathcal{O} = \emptyset$, then $\nu_{\Pi_A \setminus \Pi^\ast_A, \emptyset}^{\rm high}$ is the same as the highest term valuation studied in \cite[Section 6]{Fuj}. 
In a way similar to the proof of \cite[Theorem 6.2 (2)]{Fuj}, we deduce the following.

\begin{thm}\label{t:highest_term_valuation}
For each partition $\Pi_A \setminus \Pi^\ast_A = \mathcal{C} \sqcup \mathcal{O}$ and $\lambda \in P_+$, the Newton--Okounkov body $\Delta(G/B, \mathcal{L}_\lambda, \nu_{\mathcal{C}, \mathcal{O}}^{\rm high}, \tau_\lambda^{(\mathcal{C}, \mathcal{O})})$ is unimodularly equivalent to the Gelfand--Tsetlin polytope $GT(\lambda)$. 
In particular, it is independent of the choice of a partition $\Pi_A \setminus \Pi^\ast_A = \mathcal{C} \sqcup \mathcal{O}$ up to unimodular equivalence.
\end{thm}

The situation is quite different from the case of $\nu_{\mathcal{C}, \mathcal{O}}^{\rm low}$ since the combinatorics of a marked chain-order polytope heavily depends on the choice of a partition $\Pi_A \setminus \Pi^\ast_A = \mathcal{C} \sqcup \mathcal{O}$. 
Indeed, the Gelfand--Tsetlin polytope $GT(\lambda)$ and the FFLV polytope $FFLV(\lambda)$ have different numbers of facets in general (see \cite[Theorem 1]{Fou}).

\begin{rem}
The highest term valuation $\nu_{\mathcal{C}, \mathcal{O}}^{\rm high}$ is defined from the reduced word ${\bm i}_A$ and from a partition $\Pi_A \setminus \Pi^\ast_A = \mathcal{C} \sqcup \mathcal{O}$ which corresponds to a partition of the positions of elements in ${\bm i}_A$. 
Since we discuss Bott--Samelson varieties of general Lie type in \cite[Theorem 6.2 (2)]{Fuj}, \cref{t:highest_term_valuation} is naturally extended to an arbitrary reduced word of general Lie type and to the corresponding Bott--Samelson variety if we replace $GT(\lambda)$ with a generalized string polytope. 
\end{rem}

\section{Type $C$ case}\label{s:type_C}

In this section, we discuss the case of the symplectic group $Sp_{2n} (\c)$ (of type $C_n$).
We review some previous realizations of Gelfand--Tsetlin polytopes \cite{Oko2} and FFLV polytopes \cite{Kir2} of type $C_n$ as Newton--Okounkov bodies, and observe that our main result (\cref{t:main_thm}) cannot be naturally extended to type $C_n$ even in the case $n = 2$.
Set
\[\overline{w}_0 \coloneqq \begin{pmatrix}
0 & 0 & 0 & \cdots & -1 \\
\vdots & \vdots & \vdots & \iddots & \vdots \\
0 & 0 & 1 & \cdots & 0 \\
0 & -1 & 0 & \cdots & 0 \\
1 & 0 & 0 & \cdots & 0
\end{pmatrix} \in SL_{2n}(\c),\]
and define an algebraic group automorphism $\omega \colon SL_{2n}(\c) \xrightarrow{\sim} SL_{2n}(\c)$ by 
\[\omega(A) \coloneqq \overline{w}_0^{-1} (A^T)^{-1} \overline{w}_0\] 
for $A \in SL_{2n}(\c)$, where $A^T$ denotes the transpose of $A$.
Then the fixed point subgroup 
\[SL_{2n}(\c)^\omega \coloneqq \{A \in SL_{2n}(\c) \mid \omega(A) = A\}\]
coincides with the symplectic group  
\[Sp_{2n}(\c) \coloneqq \{A \in SL_{2n}(\c) \mid A^T \overline{w}_0 A = \overline{w}_0\}\]
with respect to the skew-symmetric matrix $\overline{w}_0$. 
In addition, the subgroup $B_{C_n} \subseteq SL_{2n}(\c)^\omega$ consisting of upper triangular matrices is a Borel subgroup of $Sp_{2n}(\c)$; hence the full flag variety of type $C_n$ is given as $Sp_{2n}(\c)/B_{C_n}$. 
Let $U^-_{C_n}$ (resp., $H_{C_n}$) denote the subgroup of $SL_{2n}(\c)^\omega$ consisting of unipotent lower triangular matrices (resp., consisting of diagonal matrices). 
Then $U^-_{C_n}$ coincides with the unipotent radical of the Borel subgroup of $Sp_{2n}(\c)$ which is opposite to $B_{C_n}$ with respect to the maximal torus $H_{C_n}$.
Let $\mathfrak{sp}_{2n}(\c)$ be the Lie algebra of $Sp_{2n}(\c)$, and $P_+$ the set of dominant integral weights for $\mathfrak{sp}_{2n}(\c)$.
For each $\lambda \in P_+$, we obtain a globally generated line bundle $\mathcal{L}_\lambda$ on $Sp_{2n}(\c)/B_{C_n}$ as in \eqref{eq:line_bundle_type_A}.  
Let us identify the set $I_{C_n}$ of vertices of the Dynkin diagram of type $C_n$ with $\{1, 2, \ldots, n\}$ as follows:
\begin{align*}
&C_n\ \begin{xy}
\ar@{=>} (50,0) *++!D{1} *\cir<3pt>{};
(60,0) *++!D{2} *\cir<3pt>{}="C"
\ar@{-} "C";(65,0) \ar@{.} (65,0);(70,0)^*!U{}
\ar@{-} (70,0);(75,0) *++!D{n-1} *\cir<3pt>{}="D"
\ar@{-} "D";(85,0) *++!D{n} *\cir<3pt>{}="E"
\end{xy}.
\end{align*}
Let $\{\varpi_i \mid i \in I_{C_n}\}$ denote the set of fundamental weights, $N(H_{C_n})$ the normalizer of $H_{C_n}$ in $Sp_{2n}(\c)$, and $W_{C_n} \coloneqq N(H_{C_n})/H_{C_n}$ the Weyl group which is generated by the set $\{s_i \mid i \in I_{C_n}\}$ of simple reflections. 
We write $N \coloneqq n^2$, and define a reduced word ${\bm i}_C = (i_1, \ldots, i_N)$ for the longest element $w_0 \in W_{C_n}$ as follows:
\begin{align*}
{\bm i}_C \coloneqq (1, \underbrace{2, 1, 2}_3, \underbrace{3, 2, 1, 2, 3}_5, \ldots, \underbrace{n, n-1, \ldots, 1, \ldots, n-1, n}_{2n-1}).
\end{align*}
For $1 \leq i, j \leq 2n$, let $E_{i, j}$ denote the $2n \times 2n$-matrix whose $(i, j)$-entry is $1$ and other entries are all $0$. 
Then we can take Chevalley generators $e_i, f_i, h_i \in \mathfrak{sp}_{2n}(\c)$, $i \in I_{C_n}$, as 
\begin{align*}
&e_1 \coloneqq E_{n, n+1},\quad f_1 \coloneqq E_{n+1, n},\quad h_1 \coloneqq E_{n, n} - E_{n+1, n+1},\\ &e_i \coloneqq E_{n-i+1, n-i+2} + E_{n+i-1, n+i},\quad f_i \coloneqq E_{n-i+2, n-i+1} + E_{n+i, n+i-1},\\ 
&h_i \coloneqq E_{n-i+1, n-i+1} - E_{n-i+2, n-i+2} + E_{n+i-1, n+i-1} - E_{n+i, n+i}
\end{align*}
for $2 \leq i \leq n$. 
For $1 \leq i \leq n$, define a lift $\overline{s}_i \in N(H_{C_n})$ for $s_i \in W_{C_n}$ by $\overline{s}_i \coloneqq \exp(f_i) \exp(-e_i) \exp(f_i)$.
For $\lambda \in P_+$ and $1 \leq k \leq n$, we write $\lambda_{\leq k} \coloneqq \sum_{1 \leq \ell \leq k} \langle \lambda, h_\ell \rangle$. 
Let us consider the \emph{Gelfand--Tsetlin poset} $(\Pi_C, \Pi^\ast_C, \lambda)$ of type $C_n$ whose Hasse diagram is given in Figure \ref{type_C_marked_Hasse}, where the circles (resp., the rectangles) denote the elements of $\Pi_C \setminus \Pi^\ast_C$ (resp., $\Pi^\ast_C$), and we write 
\[\Pi_C \setminus \Pi^\ast_C = \{q_j^{(i)} \mid 1 \leq i \leq n,\ 1 \leq j \leq 2i-1\}.\]
Note that the marking $(\lambda_a)_{a \in \Pi^\ast_C}$ is given as $(\underbrace{0, \ldots, 0}_n, \lambda_{\leq 1}, \lambda_{\leq 2}, \ldots, \lambda_{\leq n})$, which is also denoted by $\lambda$.
\begin{figure}[!ht]
\begin{center}
   \includegraphics[width=10.0cm,bb=40mm 140mm 170mm 230mm,clip]{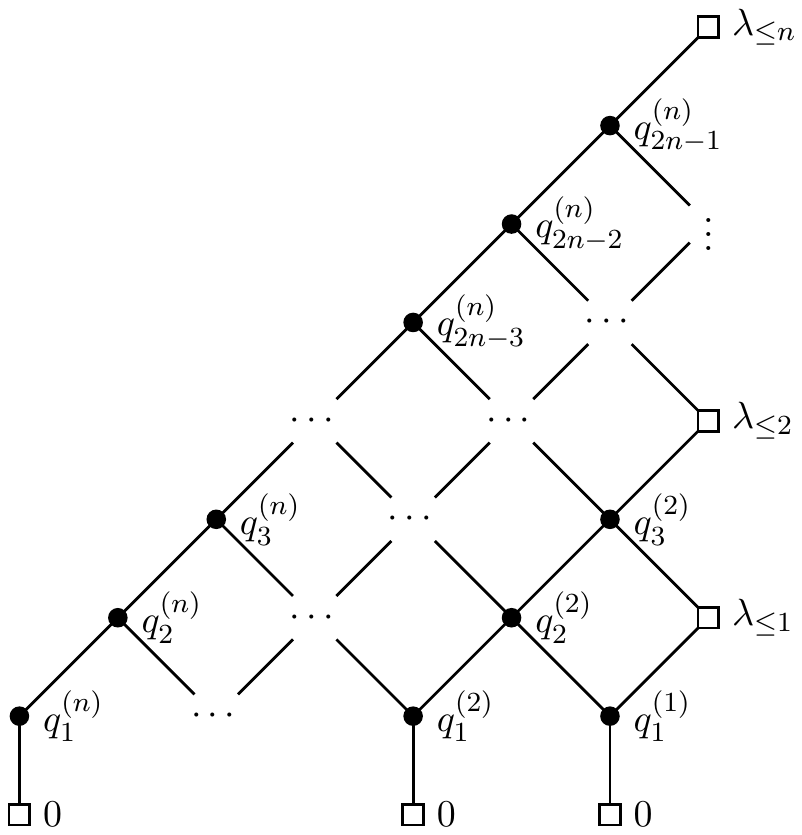}
	\caption{The marked Hasse diagram of the Gelfand--Tsetlin poset $(\Pi_C, \Pi^\ast_C, \lambda)$ of type $C_n$.}
	\label{type_C_marked_Hasse}
\end{center}
\end{figure}
By definition, the marked order polytope $\mathcal{O}(\Pi_C, \Pi^\ast_C, \lambda)$ coincides with the Gelfand--Tsetlin polytope $GT_{C_n}(\lambda)$ of type $C_n$ (see \cite[Section 6]{Lit} for the definition), and the marked chain polytope $\mathcal{C}(\Pi_C, \Pi^\ast_C, \lambda)$ coincides with the FFLV polytope $FFLV_{C_n}(\lambda)$ of type $C_n$ (see \cite[equation (2.2)]{FeFL2} for the definition). 
For all ${\bm x} = (x_{i, j} \mid 1 \leq i \leq n,\ i \leq j \leq 2n-i) \in \c^N$, there uniquely exists $(y_{i, j} \mid 1 \leq j \leq n-1,\ j+1 \leq i \leq 2n-j) \in \c^{n(n-1)}$ such that $A({\bm x}) = (a_{i, j})_{i, j} \in Sp_{2n}(\c)$, where 
\[a_{i, j} \coloneqq
\begin{cases}
x_{i, j} &\text{if}\ i \leq j \leq 2n-i,\\ 
y_{i, j} &\text{if}\ j+1 \leq i \leq 2n-j,\\ 
(-1)^i &\text{if}\ i+j = 2n+1,\\ 
0 &\text{if}\ i+j > 2n+1. 
\end{cases}\]
Then the map $\c^N \rightarrow \overline{w}_0 U^-_{C_n}$, ${\bm x} \mapsto A({\bm x})$, is an isomorphism of varieties, and the following morphism is birational:
\begin{align*}
\c^N \rightarrow Sp_{2n}(\c)/B_{C_n},\quad {\bm x} \mapsto A({\bm x}) \bmod B_{C_n}.  
\end{align*}

\begin{ex}
Let $n = 3$.
Then we have 
\[A({\bm x}) = \begin{pmatrix}
x_{1, 1} & x_{1, 2} & x_{1, 3} & x_{1, 4} & x_{1, 5} & -1 \\
y_{2, 1} & x_{2, 2} & x_{2, 3} & x_{2, 4} & 1 & 0 \\
y_{3, 1} & y_{3, 2} & x_{3, 3} & -1 & 0 & 0 \\
y_{4, 1} & y_{4, 2} & 1 & 0 & 0 & 0 \\
y_{5, 1} & -1 & 0 & 0 & 0 & 0 \\
1 & 0 & 0 & 0 & 0 & 0 
\end{pmatrix},\]
where
\begin{align*}
&y_{5, 1} = x_{1, 5}, \quad\quad\quad y_{4, 2} = x_{2, 4}, \quad\quad\quad y_{4, 1} = x_{1, 4} - x_{1, 5} x_{2, 4}, \quad\quad\quad y_{3, 2} = x_{2, 3} + x_{2, 4} x_{3, 3},\\ 
&y_{3, 1} = x_{1, 3} - x_{1, 5} x_{2, 3} + x_{1, 4} x_{3, 3} - x_{1, 5} x_{2, 4} x_{3, 3}, \quad\quad\quad y_{2, 1} = x_{1, 2} - x_{1, 5} x_{2, 2} + x_{1, 4} x_{2, 3} - x_{1, 3} x_{2, 4}.
\end{align*}
\end{ex}

We arrange the coordinates of ${\bm x}$ as 
\begin{align*}
(x_1, x_2, \ldots, x_N) \coloneqq (x_{1, 2n-1}, x_{1, 2n-2}, \ldots, x_{1, 1}, x_{2, 2n-2}, x_{2, 2n-3}, \ldots, x_{n-1, n-1}, x_{n, n}).
\end{align*}
Let us consider the lowest term valuation $\nu^{\rm low}_{x_1 > \cdots > x_N}$ on $\c(x_1, \ldots, x_N)$ with respect to the lexicographic order $x_1 > \cdots > x_N$ (see \cref{ex:lowest_and_highest_term_valuation}). 
The valuation $\nu^{\rm low}_{x_1 > \cdots > x_N}$ is slightly different from the one defined in \cite[Section 3.2]{Kir2}. 
However, the proof of \cite[Theorem 3.3]{Kir2} can also be applied to this valuation, and we obtain the following.

\begin{thm}[{see the proof of \cite[Theorem 3.3]{Kir2}}]\label{t:Kiritchenko_type_C}
Let $\lambda \in P_+$, and take a nonzero section $\tau \in H^0 (Sp_{2n}(\c)/B_{C_n}, \mathcal{L}_\lambda)$. 
Then the Newton--Okounkov body $\Delta(Sp_{2n}(\c)/B_{C_n}, \mathcal{L}_\lambda, \nu^{\rm low}_{x_1 > \cdots > x_N}, \tau)$ coincides with the FFLV polytope $FFLV_{C_n}(\lambda)$ up to translations by integer vectors.
\end{thm} 

Arrange the elements of $\Pi_C \setminus \Pi^\ast_C$ as 
\begin{equation}\label{eq:arrangement_of_poset_type_C}
\begin{aligned}
(q_1, q_2, \ldots, q_N) \coloneqq (q_1^{(1)}, q_2^{(2)}, q_1^{(2)}, q_3^{(2)}, q_2^{(3)}, \ldots, q_{2n-3}^{(n-1)}, q_2^{(n)}, \ldots, q_n^{(n)}, q_1^{(n)}, q_{n+1}^{(n)}, \ldots, q_{2n-1}^{(n)}).
\end{aligned}
\end{equation}
Then a partition $\Pi_C \setminus \Pi^\ast_C = \mathcal{C} \sqcup \mathcal{O}$ is identified with a partition of the set $\{1, 2, \ldots, N\}$; we regard this set $\{1, 2, \ldots, N\}$ as the set of positions of entries in ${\bm i}_C$. 
Using the arrangement \eqref{eq:arrangement_of_poset_type_C}, we identify $\r^{\Pi_C \setminus \Pi^\ast_C}$ with $\r^N$.
As in type $A_n$ case, each partition $\Pi_C \setminus \Pi^\ast_C = \mathcal{C} \sqcup \mathcal{O}$ gives a map $\Omega_{\mathcal{C}, \mathcal{O}} \colon \c^N \rightarrow Sp_{2n}(\c)$, which induces a birational morphism 
\[\widehat{\Omega}_{\mathcal{C}, \mathcal{O}} \colon \c^N \rightarrow Sp_{2n}(\c)/B_{C_n},\quad {\bm t} \mapsto \Omega_{\mathcal{C}, \mathcal{O}}({\bm t}) \bmod B_{C_n}.\]
Using this birational morphism, we identify the function field $\c(Sp_{2n}(\c)/B_{C_n})$ with the field $\c(t_1, \ldots, t_N)$ of rational functions in $t_1, \ldots, t_N$. 
For each permutation $\sigma \in \mathfrak{S}_{N}$, let $\nu^{(\mathcal{C}, \mathcal{O})}_{t_{\sigma(1)} > \cdots > t_{\sigma(N)}}$ denote the lowest term valuation $\nu^{\rm low}_{t_{\sigma(1)} > \cdots > t_{\sigma(N)}}$ on $\c(t_{\sigma(1)}, \ldots, t_{\sigma(N)})$ with respect to the lexicographic order $t_{\sigma(1)} > \cdots > t_{\sigma(N)}$ (see \cref{ex:lowest_and_highest_term_valuation}).
As in type $A_n$ case, the Gelfand--Tsetlin polytopes $GT_{C_n}(\lambda)$ and the FFLV polytopes $FFLV_{C_n}(\lambda)$ for $\lambda \in P_+$ can be realized as Newton--Okounkov bodies associated with valuations of the form $\nu^{(\mathcal{C}, \mathcal{O})}_{t_{\sigma(1)} > \cdots > t_{\sigma(N)}}$.
We rearrange the coordinates of ${\bm t}$ as 
\begin{align*}
(t_1^\prime, t_2^\prime, \ldots, t_N^\prime) \coloneqq (t_N, t_{N-1}, \ldots, t_{N-n+2}, t_{N-n}, \ldots, t_{(n-1)^2+1}, t_{N-n+1}, t_{(n-1)^2}, \ldots, t_4, t_2, t_3, t_1).
\end{align*}
If $\mathcal{C} = \Pi_C \setminus \Pi^\ast_C$ and $\mathcal{O} = \emptyset$, then it follows that $\Omega_{\Pi_C \setminus \Pi^\ast_C, \emptyset} ({\bm t}) \in \overline{w}_0 U^-_{C_n}$ for all ${\bm t} \in \c^N$, and the map $\c^N \rightarrow \overline{w}_0 U^-_{C_n}$, ${\bm t} \mapsto \Omega_{\Pi_C \setminus \Pi^\ast_C, \emptyset}({\bm t})$, is an isomorphism of varieties. 
In addition, if $A({\bm x}) = \Omega_{\Pi_C \setminus \Pi^\ast_C, \emptyset}({\bm t})$, then we see that 
\begin{equation}\label{eq:transition_type_C_PBW}
\begin{aligned}
x_\ell = 
\begin{cases} 
(-1)^n (t_\ell^\prime + \sum_{1 \leq c \leq k} (-1)^{c-1} t_{\ell-k+c-1}^\prime t_{\ell-k-c}^\prime)&\text{if}\ N-\ell = k^2\ \text{for some}\ 0 \leq k < n,\\ 
(-1)^\ell t_\ell^\prime &\text{otherwise}
\end{cases}
\end{aligned}
\end{equation}
for $1 \leq \ell \leq N$.

\begin{ex}
Let $n = 3$.
Then we have 
\[\Omega_{\Pi_C \setminus \Pi^\ast_C, \emptyset}({\bm t}) = \begin{pmatrix}
-t_7 - t_6 t_8 + t_5 t_9 & t_5 & -t_6 & t_8 & -t_9 & -1 \\
t_5 + t_4 t_6 - t_2 t_8 - t_3 t_9 - t_2 t_4 t_9 & -t_3 - t_2 t_4 & -t_2 & t_4 & 1 & 0 \\
-t_6 - t_1 t_8 - t_2 t_9 - t_1 t_4 t_9 & -t_2 - t_1 t_4 & -t_1 & -1 & 0 & 0 \\
t_8 + t_4 t_9 & t_4 & 1 & 0 & 0 & 0 \\
-t_9 & -1 & 0 & 0 & 0 & 0 \\
1 & 0 & 0 & 0 & 0 & 0 
\end{pmatrix}.\]
\end{ex}

\begin{lem}\label{l:change_of_coordinate_from_x_to_t}
For each $f \in \c(Sp_{2n}(\c)/B_{C_n}) \setminus \{0\}$, it holds that 
\[\nu^{(\Pi_C \setminus \Pi^\ast_C, \emptyset)}_{t_1^\prime > \cdots > t_N^\prime}(f) = \nu^{\rm low}_{x_1 > \cdots > x_N}(f).\]
\end{lem}

\begin{proof}
By the definition of valuations, it suffices to prove the equality for every nonzero polynomial 
\[f = \sum_{{\bm a} = (a_1, \ldots, a_N) \in \z_{\geq 0}^N} c_{\bm a} x_1^{a_1} \cdots x_N^{a_N} \in \c[x_1, \ldots, x_N] \setminus \{0\},\]
where $c_{\bm a} \in \c$. 
For each $(a_1, \ldots, a_N) \in \z_{\geq 0}^N$, we deduce by \eqref{eq:transition_type_C_PBW} that 
\begin{align*}
\nu^{(\Pi_C \setminus \Pi^\ast_C, \emptyset)}_{t_1^\prime > \cdots > t_N^\prime} (x_1^{a_1} \cdots x_N^{a_N}) &= (a_1, \ldots, a_N)\\
&= \nu^{\rm low}_{x_1 > \cdots > x_N}(x_1^{a_1} \cdots x_N^{a_N}), 
\end{align*}
which implies that $\nu^{(\Pi_C \setminus \Pi^\ast_C, \emptyset)}_{t_1^\prime > \cdots > t_N^\prime}(f) = \nu^{\rm low}_{x_1 > \cdots > x_N}(f)$. 
This proves the lemma.
\end{proof}

The following is an immediate consequence of \cref{t:Kiritchenko_type_C} and \cref{l:change_of_coordinate_from_x_to_t}.

\begin{cor}
For each $\lambda \in P_+$ and $\tau \in H^0 (Sp_{2n}(\c)/B_{C_n}, \mathcal{L}_\lambda) \setminus \{0\}$, the Newton--Okounkov body $\Delta(Sp_{2n}(\c)/B_{C_n}, \mathcal{L}_\lambda, \nu^{(\Pi_C \setminus \Pi^\ast_C, \emptyset)}_{t_1^\prime > \cdots > t_N^\prime}, \tau)$ coincides with the FFLV polytope $FFLV_{C_n}(\lambda)$ up to translations by integer vectors.
\end{cor}

Let us consider the case $\mathcal{C} = \emptyset$ and $\mathcal{O} = \Pi_C \setminus \Pi^\ast_C$. 
In this case, Okounkov \cite{Oko2} realized $GT_{C_n}(\lambda)$ using the valuation $\nu^{(\emptyset, \Pi_C \setminus \Pi^\ast_C)}_{t_N > \cdots > t_1}$ as follows.

\begin{thm}[{\cite[Theorem 2]{Oko2}}]
For each $\lambda \in P_+$ and $\tau \in H^0 (Sp_{2n}(\c)/B_{C_n}, \mathcal{L}_\lambda) \setminus \{0\}$, the Newton--Okounkov body $\Delta(Sp_{2n}(\c)/B_{C_n}, \mathcal{L}_\lambda, \nu^{(\emptyset, \Pi_C \setminus \Pi^\ast_C)}_{t_N > \cdots > t_1}, \tau)$ is unimodularly equivalent to the Gelfand--Tsetlin polytope $GT_{C_n}(\lambda)$. 
\end{thm}

However, \cref{t:main_thm} cannot be naturally extended to $Sp_{2n}(\c)$ even in the case $n = 2$ as we see below. 
Let $n = 2$, and take $\rho \in P_+$ to be $\varpi_1 + \varpi_2$. 
We consider the following three partitions of $\Pi_C \setminus \Pi^\ast_C = \{q_1, \ldots, q_4\}$:
\begin{align*}
&(\mathcal{C}_1, \mathcal{O}_1) \coloneqq (\emptyset, \Pi_C \setminus \Pi^\ast_C), &&(\mathcal{C}_2, \mathcal{O}_2) \coloneqq (\Pi_C \setminus \Pi^\ast_C, \emptyset), &&(\mathcal{C}_3, \mathcal{O}_3) \coloneqq (\{q_3, q_4\}, \{q_1, q_2\}).
\end{align*}
Then the marked chain-order polytope $\Delta_{(\mathcal{C}_1, \mathcal{O}_1)} (\Pi_C, \Pi^\ast_C, \rho)$ (resp., $\Delta_{(\mathcal{C}_2, \mathcal{O}_2)} (\Pi_C, \Pi^\ast_C, \rho)$) is the Gelfand--Tsetlin polytope $GT_{C_2}(\rho)$ (resp., the FFLV polytope $FFLV_{C_2}(\rho)$).
In addition, the marked chain-order polytope $\Delta_{(\mathcal{C}_3, \mathcal{O}_3)} (\Pi_C, \Pi^\ast_C, \rho)$ is the set of $(a_1, \ldots, a_4) \in \r_{\geq 0}^4$ satisfying the following inequalities:
\[a_1 \leq 1,\quad a_4 \leq 1,\quad a_1 \leq a_2 \leq 2 -a_4,\quad a_3 \leq a_2.\]
Then the marked chain-order polytopes $\Delta_{(\mathcal{C}_k, \mathcal{O}_k)} (\Pi_C, \Pi^\ast_C, \rho)$, $1 \leq k \leq 3$, are all unimodularly equivalent to each other.
For each $1 \leq k \leq 3$, define $\tau^{(\mathcal{C}_k, \mathcal{O}_k)}_{\rho} \in H^0(Sp_4(\c)/B_{C_2}, \mathcal{L}_{\rho})$ as in \eqref{eq:specific_section_tau}. 
Then the Newton--Okounkov bodies $\Delta(Sp_4(\c)/B_{C_2}, \mathcal{L}_{\rho}, \nu_{t_{\sigma(1)} > \cdots > t_{\sigma(4)}}^{(\mathcal{C}_k, \mathcal{O}_k)}, \tau^{(\mathcal{C}_k, \mathcal{O}_k)}_{\rho})$ for $1 \leq k \leq 3$ and $\sigma \in \mathfrak{S}_4$ are given in Table \ref{type_C_NO_body}.
In this table, the Newton--Okounkov bodies are divided into four types: $GT$, $NZ$, $\Delta$, and $\times$.
We mean by ``$GT$'' that the Newton--Okounkov body is unimodularly equivalent to the Gelfand--Tsetlin polytope $GT_{C_2}(\rho)$. 
In addition, the type ``$NZ$'' implies that the Newton--Okounkov body is unimodularly equivalent to the Nakashima--Zelevinsky polytope $NZ_{{\bm i}_C}(\rho)$ associated with $\rho$ and ${\bm i}_C = (1, 2, 1, 2)$ (see \cite[Example 5.10]{FN}). 
More concretely, the Nakashima--Zelevinsky polytope $NZ_{{\bm i}_C}(\rho)$ is given as the set of $(a_1, \ldots, a_4) \in \r_{\geq 0}^4$ satisfying the following inequalities:
\[a_4 \leq 1,\quad a_3 \leq a_4 +1,\quad a_2 \leq \min\{a_3 +1, 2a_3\},\quad 2a_1 \leq \min\{a_2, 2\}.\]
This polytope $NZ_{{\bm i}_C}(\rho)$ has $11$ vertices, and hence it is not unimodularly equivalent to the Gelfand--Tsetlin polytope $GT_{C_2}(\rho)$ which has $12$ vertices.    
We mean by ``$\Delta$'' that the Newton--Okounkov body coincides with the set of $(a_1, \ldots, a_4) \in \r_{\geq 0}^4$ satisfying the following inequalities:
\[a_2 \leq 1,\quad a_4 \leq 1,\quad a_1 \leq 1 +a_4 -a_2,\quad a_3 \leq 2 -a_2 -a_4.\]
This is an integral convex polytope with $12$ vertices, but it is not unimodularly equivalent to the Gelfand--Tsetlin polytope $GT_{C_2}(\rho)$. 
Finally, the type ``$\times$'' means that the Newton--Okounkov body $\Delta(Sp_4(\c)/B_{C_2}, \mathcal{L}_{\rho}, \nu_{t_{\sigma(1)} > \cdots > t_{\sigma(4)}}^{(\mathcal{C}_k, \mathcal{O}_k)}, \tau^{(\mathcal{C}_k, \mathcal{O}_k)}_{\rho})$ is strictly bigger than the convex hull of $16$ points in the set 
\[\{\nu_{t_{\sigma(1)} > \cdots > t_{\sigma(4)}}^{(\mathcal{C}_k, \mathcal{O}_k)}(\tau/\tau^{(\mathcal{C}_k, \mathcal{O}_k)}_{\rho}) \mid \tau \in H^0(Sp_4(\c)/B_{C_2}, \mathcal{L}_{\rho}) \setminus \{0\}\};\]
see also \cref{p:property_valuation} (2).
More precisely, this convex hull has volume $\frac{5}{6}$ for this type $\times$ while the volume of the Newton--Okounkov body is $1$. 
For the previous three types $GT$, $NZ$, and $\Delta$, the Newton--Okounkov body is precisely the convex hull of such $16$ points. 
Summarizing, Newton--Okounkov bodies with different types are not unimodularly equivalent to each other. 
In particular, we see by Table \ref{type_C_NO_body} that the Newton--Okounkov body $\Delta(Sp_4(\c)/B_{C_2}, \mathcal{L}_{\rho}, \nu_{t_{\sigma(1)} > \cdots > t_{\sigma(4)}}^{(\mathcal{C}_3, \mathcal{O}_3)}, \tau^{(\mathcal{C}_3, \mathcal{O}_3)}_{\rho})$ is not unimodularly equivalent to $\Delta_{(\mathcal{C}_3, \mathcal{O}_3)} (\Pi_C, \Pi^\ast_C, \rho)$ for any $\sigma \in \mathfrak{S}_4$.
This implies that \cref{t:main_thm} cannot be naturally extended to the marked chain-order polytope $\Delta_{(\mathcal{C}_3, \mathcal{O}_3)} (\Pi_C, \Pi^\ast_C, \rho)$. 

\begin{table}[h]
\begin{tabular}{|c||c|c|c|} \hline
lexicographic orders & $(\mathcal{C}_1, \mathcal{O}_1)$ & $(\mathcal{C}_2, \mathcal{O}_2)$ & $(\mathcal{C}_3, \mathcal{O}_3)$ \\ \hline
$t_1 > t_2 > t_3 > t_4$ & $NZ$ & $NZ$ & $NZ$ \\ \hline
$t_1 > t_2 > t_4 > t_3$ & $NZ$ & $NZ$ & $NZ$ \\ \hline
$t_1 > t_3 > t_2 > t_4$ & $NZ$ & $\times$ & $\times$ \\ \hline
$t_1 > t_3 > t_4 > t_2$ & $\times$ & $\times$ & $\times$ \\ \hline
$t_1 > t_4 > t_2 > t_3$ & $\times$ & $NZ$ & $NZ$ \\ \hline
$t_1 > t_4 > t_3 > t_2$ & $\times$ & $NZ$ & $NZ$ \\ \hline
$t_2 > t_1 > t_3 > t_4$ & $NZ$ & $\Delta$ & $NZ$ \\ \hline
$t_2 > t_1 > t_4 > t_3$ & $NZ$ & $\Delta$ & $NZ$ \\ \hline
$t_2 > t_3 > t_1 > t_4$ & $GT$ & $\Delta$ & $NZ$ \\ \hline
$t_2 > t_3 > t_4 > t_1$ & $GT$ & $\Delta$ & $NZ$ \\ \hline
$t_2 > t_4 > t_1 > t_3$ & $NZ$ & $\Delta$ & $NZ$ \\ \hline
$t_2 > t_4 > t_3 > t_1$ & $GT$ & $\Delta$ & $NZ$ \\ \hline
$t_3 > t_1 > t_2 > t_4$ & $GT$ & $\times$ & $\times$ \\ \hline
$t_3 > t_1 > t_4 > t_2$ & $GT$ & $\times$ & $\times$ \\ \hline
$t_3 > t_2 > t_1 > t_4$ & $GT$ & $NZ$ & $\times$ \\ \hline
$t_3 > t_2 > t_4 > t_1$ & $GT$ & $NZ$ & $\times$ \\ \hline
$t_3 > t_4 > t_1 > t_2$ & $GT$ & $\times$ & $\times$ \\ \hline
$t_3 > t_4 > t_2 > t_1$ & $GT$ & $\times$ & $\times$ \\ \hline
$t_4 > t_1 > t_2 > t_3$ & $NZ$ & $NZ$ & $NZ$ \\ \hline
$t_4 > t_1 > t_3 > t_2$ & $NZ$ & $NZ$ & $NZ$ \\ \hline
$t_4 > t_2 > t_1 > t_3$ & $NZ$ & $GT$ & $NZ$ \\ \hline
$t_4 > t_2 > t_3 > t_1$ & $GT$ & $GT$ & $NZ$ \\ \hline
$t_4 > t_3 > t_1 > t_2$ & $GT$ & $NZ$ & $NZ$ \\ \hline
$t_4 > t_3 > t_2 > t_1$ & $GT$ & $NZ$ & $NZ$ \\ \hline
\end{tabular}
\vspace{4mm}
   \caption{The Newton--Okounkov bodies $\Delta(Sp_4(\c)/B_{C_2}, \mathcal{L}_{\rho}, \nu_{t_{\sigma(1)} > \cdots > t_{\sigma(4)}}^{(\mathcal{C}_k, \mathcal{O}_k)}, \tau^{(\mathcal{C}_k, \mathcal{O}_k)}_{\rho})$ for $1 \leq k \leq 3$ and $\sigma \in \mathfrak{S}_4$.}
	\label{type_C_NO_body}
\end{table}

\end{document}